\documentclass[11pt,a4paper]{amsart} %
\usepackage[left=25mm,top=30mm,bottom=25mm,right=25mm]{geometry}
\usepackage{amssymb,dsfont}
\usepackage[all]{xy}
\usepackage{bm}
\usepackage{calrsfs}
\usepackage{tensor}
\usepackage[utf8]{inputenc}                     
\usepackage{cmap}                               
\usepackage[english]{babel}
\usepackage[fixlanguage]{babelbib}
\usepackage[pdfdisplaydoctitle = true,
      colorlinks = true,
      urlcolor = blue,
      citecolor = blue,
      linkcolor = blue,
      pdfstartview = FitH,
      pdfpagemode = UseNone,
      bookmarksnumbered = true,
      unicode=true]{hyperref}           
\usepackage{hyperxmp}

\newtheorem{thm}{Theorem}[section]

\newtheorem{lem}[thm]{Lemma}
\newtheorem{prop}[thm]{Proposition}

\theoremstyle{definition}
\newtheorem{defn}[thm]{Definition}
\newtheorem{exam}[thm]{Example}

\theoremstyle{remark}
\newtheorem{rem}[thm]{Remark}

\numberwithin{equation}{section}

\newcommand{\RR}{\mathbb{R}}                

\newcommand{\TT}{\mathbb{T}}                
\newcommand{\EE}{\mathbb{E}}                
\newcommand{\espace}{\mathcal{E}}           
\newcommand{\body}{\mathcal{B}}             
\newcommand{\SO}{\mathrm{SO}}               
\newcommand{\Diff}{\mathrm{Diff}}           
\newcommand{\Isom}{\mathrm{Isom}}           
\newcommand{\Met}{\mathrm{Met}}             
\newcommand{\Vol}{\mathrm{Vol}}             
\newcommand{\Emb}{\mathrm{Emb}^{\infty}}    
\newcommand{\Vect}{\mathrm{Vect}}
\newcommand{\End}{\mathrm{End}}
\newcommand{\Cinf}{\mathrm{C}^{\infty}}

\newcommand{\bepsilon}{{\bm{\varepsilon}}}
\newcommand{\bbeta}{\bm{\eta}}
\newcommand{\bgamma}{{\bm{\gamma}}}

\newcommand{\bsigma}{{\bm{\sigma}}}
\newcommand{\btau}{{\bm{\tau}}}
\newcommand{\btheta}{{\bm{\theta}}}
\newcommand{\bomega}{{\bm{\omega}}}
\newcommand{\bOmega}{{\bm{\Omega}}}
\newcommand{\bUpsilon}{{\bm{\Upsilon}}}
\newcommand{\bLambda}{{\bm{\Lambda}}}

\newcommand{\set}[1]{\left\{#1\right\}}

\newcommand{\bq}{\mathbf{q}}                
\newcommand{\vol}{\mathrm{vol}}             
\newcommand{\bt}{\mathbf{t}}
\newcommand{\bU}{\mathbf{U}}                
\newcommand{\bV}{\mathbf{V}}                
\newcommand{\bA}{\mathbf{A}}                

\newcommand{\bM}{\mathbf{M}}                

\newcommand{\XX}{\mathbf{X}}                
\newcommand{\bC}{\mathbf{C}}
\newcommand{\bD}{\mathbf{D}}
\newcommand{\bE}{\mathbf{E}}

\newcommand{\bL}{\mathbf{L}}

\newcommand{\bR}{\mathbf{R}}
\newcommand{\bS}{\mathbf{S}}
\newcommand{\bT}{\mathbf{T}}

\newcommand{\msigma}{{\pmb{\mathfrak{S}}}}

\newcommand{\xx}{\mathbf{x}}                
\newcommand{\yy}{\mathbf{y}}

\newcommand{\vv}{\bm{v}}
\newcommand{\ww}{\bm{w}}                    
\newcommand{\ba}{\mathbf{a}}
\newcommand{\bb}{\mathbf{b}}
\newcommand{\bc}{\mathbf{c}}
\newcommand{\bd}{\mathbf{d}}

\newcommand{\bk}{\mathbf{k}}
\newcommand{\bs}{\mathbf{s}}
\newcommand{\bw}{\mathbf{w}}

\newcommand{\WW}{\bm{W}}                    
\newcommand{\VV}{\bm{V}}                    
\newcommand{\UU}{\bm{U}}                    
\newcommand{\uu}{\bm{u}}                    
\newcommand{\pp}{p}                         
\newcommand{\ee}{\bm{e}}                    
\newcommand{\bF}{\mathbf{F}}                

\newcommand{\Id}{\mathbf{Id}}

\newcommand{\dd}[2]{\frac{d{#1}}{d{#2}}}

\newcommand{\dcov}[2]{\frac{D_{#1}\,{#2}}{Dt}}
\newcommand{\dmat}[2]{\frac{d_{#1}\,{#2}}{dt}}

\DeclareMathOperator{\Lie}{L} %
\DeclareMathOperator{\tr}{tr} %
\DeclareMathOperator{\dive}{\mathbf{div}} %
\DeclareMathOperator{\Log}{Log} %
\DeclareMathOperator{\Exp}{\mathrm{Exp}} %

\hypersetup{
    pdftitle={Objective rates as covariant derivatives on the manifold of Riemannian metrics},
    pdfauthor={Boris Kolev and Rodrigue Desmorat},
    pdfsubject={MSC 2020 : 74B20; 58D17; 74A05; 74A20},
    pdfkeywords={Objective derivatives; Material time rate; Nonlinear elasticity; Manifold of Riemannian metrics; Geometric mechanics; Constitutive laws},
    pdflang=en
}
\begin{document}

\title[Objective rates as covariant derivatives]{Objective rates as covariant derivatives \protect\\ on the manifold of Riemannian metrics}%

\author{B. Kolev}
\address[Boris Kolev]{Université Paris-Saclay, ENS Paris-Saclay, CentraleSupélec, CNRS, LMPS - Laboratoire de Mécanique Paris-Saclay, 91190, Gif-sur-Yvette, France}
\email{boris.kolev@ens-paris-saclay.fr}

\author{R. Desmorat}
\address[Rodrigue Desmorat]{Université Paris-Saclay, ENS Paris-Saclay, CentraleSupélec, CNRS, LMPS - Laboratoire de Mécanique Paris-Saclay, 91190, Gif-sur-Yvette, France}
\email{rodrigue.desmorat@ens-paris-saclay.fr}

\date{\today}%
\subjclass[2020]{74B20; 58D17; 74A05; 74A20}
\keywords{Objective derivatives; Material time rate; Nonlinear elasticity; Manifold of Riemannian metrics; Geometric mechanics; Constitutive laws}%
\dedicatory{This work is dedicated to Professor Paul Rougée.}

\begin{abstract}
  The subject of so-called objective derivatives in Continuum Mechanics has a long history and has generated varying views concerning their true mathematical interpretation. Several attempts have been made to provide a mathematical definition that would at least partially unify the existing notions. In this paper, we demonstrate that, under natural assumptions, all objective derivatives correspond to covariant derivatives on the infinite-dimensional manifold $\mathrm{Met}(\mathcal{B})$ of Riemannian metrics on the body. Furthermore, a natural Leibniz rule enables canonical extensions from covariant to contravariant tensor fields and vice versa. This makes the sometimes-used distinction between objective derivatives of ``Lie type'' and ``co-rotational type'' unnecessary. For an exhaustive list of objective derivatives found in the literature, we exhibit the corresponding covariant derivative on $\mathrm{Met}(\mathcal{B})$.
\end{abstract}

\maketitle

\begin{small}
  \tableofcontents
\end{small}

\section{Introduction}

The re-foundation and geometrization of Continuum Mechanics was initiated by C. Truesdell and W. Noll at the beginning of the second half of the 20th century~\cite{TN1965,WT1973}. These authors have formulated a frame independent theory, using in a systematic way, intrinsic notations and defining general axioms -- among them the principle of objectivity -- from which soundly derives the finite strain theory. The starting point of this formulation is an abstract manifold of dimension 3 (with boundary) $\body$, called the \emph{body}, and \emph{equipped with a volume form} (the mass measure). The usual \emph{reference configuration} $\Omega_{0}$ is thus just the choice of a particular embedding of $\body$ into the Euclidean space $\espace$ (which usually represents the unloaded system), while the \emph{deformed/actual configuration} is an embedding of $\body$ into $\espace$ (which represents the configuration of the system, in equilibrium, after a certain loading has been applied). The geometrization of Continuum Mechanics was continued afterwards, under the impulse of Marsden and Hughes~\cite{MH1994}, and is still alive today
\cite{ES1980,SM1984,San1992,ST1994,Sve1995,GK1996,SH1997,Kad1999,NS2010,Dim2011,Fia2011,Ber2012,Ste2015,SE2020}.
The possibility to make mechanical calculations (and numerical discretizations in~\cite{Fia2015,Fia2016}) directly on the body $\body$ emerges with the formulation, by Noll~\cite{Nol1972} and later by Rougée~\cite{Rou1980}, of so-called \emph{intrinsic stresses}, and by the possibility to recast boundary conditions on the abstract manifold $\body$ (see Noll's formulation in~\cite{Nol1978}). In line with Eringen~\cite{Eri1962}, Green and Zerna~\cite{GZ1968}, Benzecri~\cite{Ben1967}, Noll~\cite{Nol1972,Nol1978}, and then Epstein and Segev~\cite{ES1980,Seg1986}, Rougée~\cite{Rou1980,Rou1991a,Rou1997,Rou2006} has furthermore rightly understood the fundamental role played in Continuum Mechanics by the \emph{manifold of Riemannian metrics} on the body $\body$. This infinite dimensional manifold, noted $\Met(\body)$, can itself be endowed with a Riemannian structure, in fact an $L^{2}$-metric. This is well-known in the framework of general relativity, where such a metric (on the manifold of Lorentzian metrics) called the \emph{Ebin metric}~\cite{Ebi1968} has been introduced. The geometry of the manifold of Riemannian metrics has been extensively studied for its own sake by many authors~\cite{Ebi1967,FG1989,GM1991,Cla2018,Cla2010,Cla2010a,Cla2013,BHM2013}. The subject is connected to General Relativity but also to the theory of the Ricci flow~\cite{Ham1982a,Ham1986} which has become famous following the work of Grigori Perelman~\cite{Per2002,Per2003,Per2003a}. It is perhaps less popular in Continuum Mechanics, and we must recognize Paul Rougée~\cite{Rou1980,Rou1991,Rou1991a,Rou1997,Rou2006} as a pioneer for having introduced such a concept in this field.

The \emph{objectivity principle} (\emph{i.e.} material frame indifference~\cite{Old1950,Nol1959,TN1965}) is nowadays a cornerstone for the formulation of rate-form constitutive equations for solids and fluids. So called objective time-derivatives (rates) of objective mechanical quantities have been proposed in the literature~\cite{Zar1903,Jau1911,Old1950,Tru1955,GN1965}, in order to introduce some kind of elasticity for viscous fluids or to derive computationally efficient formulations of finite strain elasto-(visco-)plasticity~\cite{Lad1980,Lad1999,Lub1990}. Some extensions to four-dimensional formalism can be found in~\cite{RPK2013,PRA2015}. Since a lot of objective derivatives are available in the literature, the natural question arose of how to unify them all~\cite{Hil1978}, as well as of better understanding their intrinsic nature~\cite{MH1994}. In other words, to clarify the mathematical concept underlined. A fuzzy classification into \emph{co-rotational types} (generalizing the Green--Naghdi objective rate), and \emph{non co-rotational types} (of Lie--Olroyd--Truesdell type) was sometimes made but is still unclear. Xiao and coworkers~\cite{XBM1998} have then proposed a quite general expression which includes all known co-rotational objective derivatives, but did not explain what kind of mathematical object they were. The remaining question was thus to find the underlying rigorous mathematical formulation of the concept that allows to unify them all (and eventually formulate new ones).

Marsden and Hughes did write that \emph{``All so-called objective rates of second-order tensors are in fact Lie derivatives''} \cite[Box 6.1 p. 99]{MH1994}, or more precisely, linear combinations of Lie derivatives of the (contravariant) stress tensor and of its different covariant and mixed forms (obtained by lowering subscripts thanks to the Euclidean metric). However, these ``Lie derivatives'' do not include Fiala's objective rate~\cite{Fia2004}, derived a few years later, and their statement is thus not a satisfying answer to the question.

Therefore, if objective rates are not Lie derivative, what are they ? The aim of this paper is to demonstrate that all of them are in fact covariant derivatives on $\Met(\body)$. This is how the role of the manifold of Riemannian metrics $\Met(\body)$ becomes important here. Rougée~\cite{Rou1991,Rou1991a,Rou2006} was the first to understand that the Jaumann objective rate was in fact a covariant derivative on $\Met(\body)$. Following this idea, Fiala proposed in~\cite{Fia2004} a new objective rate, which was formulated as the Riemannian covariant derivative on $\Met(\body)$ corresponding to Ebin's metric. However, in a recent paper~\cite{Fia2019}, he seems to differentiate two kinds of objective derivatives: one for strains of covariant derivative type, and one for stresses of Lie type. We don't follow this point of view. We first formulate objective derivatives for strains as covariant derivatives on the tangent bundle $T\Met(\body)$. Formally, such a covariant derivative induces a covariant derivative on the cotangent bundle $T^{\star}\Met(\body)$, which is a space of tensor-distributions. If this covariant derivative preserves tensor-distributions with density, then it induces an objective derivative for stresses. This observation seems to be new. In practice, this happens to be the case for all known objective derivatives.

The aim of the present work is twofold. First, it is to introduce in Continuum Mechanics the manifold $\Met(\body)$ of Riemannian metrics on the body $\body$ and to show that \emph{each covariant derivative on $\Met(\body)$ induces an objective derivative} on symmetric second-order covariant tensor fields (strains). Then, using Leibniz rule, we prove that, under certain conditions, this covariant derivative induces also an objective derivative on symmetric second-order contravariant tensor fields (stresses). These conditions are always satisfied when the objective derivative is \emph{local} (meaning that it depends only on the first jet of the involved variables). Moreover, in that case, we also prove the converse: each local objective derivative derives from a covariant derivative on $\Met(\body)$.

Finally, we illustrate our general statement by showing how the objective rates of the literature, including Green--Naghdi's rate, Hill's, Marsden--Hughes' and Xiao--Bruhns-Meyers' families, as well as Fiala's rate, interpret this way. Meanwhile, it was necessary to reconsider the very definition of objectivity and precise this concept from the mathematical point of view. Our discourse is focused on Continuum Mechanics but uses rather sophisticated notions from differential geometry (in infinite dimension). For the sake of self-completeness, the essential concepts are recalled in the appendices. We emphasise that, following Arnold~\cite{Arn1966,AK1998}, our goal is to concentrate on an intuitive geometric approach, mimicking what is known in finite dimensional differential geometry: by not getting into functional analysis too much but rather by working with smooth functions and formal calculations, which have to be verified on a case-by-case basis. Our goal is not to start by constructing a consistent theory of infinite-dimensional differential geometry, which encounters serious analytical difficulties. For a discussion of these questions, we redirect the readers to the following works~\cite{EM1970,Mic1980,Ham1982,Mil1984,FG1989,KM1997,IKT2013,SE2020}. Furthermore, we have chosen to work in the smooth category. It might however be possible to work in the category $C^{p}$, where $p \ge 1$, as in \cite{Seg1986,SE2020} for instance.

The outline of the paper is as follows. In~\autoref{sec:MMC}, we recall the formalism introduced by Noll to model a continuum medium. In this framework, a configuration is represented by an embedding of the body $\body$ into the Euclidean space $\espace$. In~\autoref{sec:strain-and-stress}, we recall basic materials and formulas concerning strains and stresses. The manifold of Riemannian metrics $\Met(\body)$ is introduced in~\autoref{sec:manifold-of-metrics}, where all the mathematical concepts associated with it are introduced: the Riemannian structure, covariant derivatives, and the Riemannian exponential mapping. In~\autoref{sec:objectivity}, we formulate a rigorous mathematical definition of objectivity and illustrate it by examples. Then, in~\autoref{sec:material-derivatives}, we extend this formulation to time derivatives and prove that every covariant derivative on $\Met(\body)$ defines an objective derivative. The converse is proved in~\autoref{sec:converse-theorem}, under local hypothesis. Finally in~\autoref{sec:objective-rates}, we show that all objective derivatives from the literature are special occurrences of this geometric formalism. Besides, for convenience of the reader, we have added four appendices to summarize the main mathematical concepts and formulas used in this paper, in order to fix the notations and to be as self-contained as possible.

\section{The configuration space in Continuum Mechanics}
\label{sec:MMC}

In Continuum Mechanics, the ambient space $\espace$ is represented by a three-dimensional Euclidean affine space. Designating by $\bq$ the Euclidean metric on $\espace$, it is better to consider this space as a Riemannian manifold $(\espace,\bq)$ and forget, at first, this affine structure of space. In a pictorial way, we can consider $(\espace,\bq)$\footnote{To be more accurate, the manifold $\espace$ is not ``the space'' but the \emph{typical fiber}, of dimension 3, of a fibered manifold of dimension $4$ with a Galilean structure~\cite{Kue1972,Kue1976,DK1977}.} such as an observer's three-dimensional ``laboratory notebook'' (or reference system) in which he records his observations, after choosing a unit length (the metric). The material medium is parameterized by a three-dimensional compact and orientable manifold with boundary, noted $\body$, and called the \emph{body}. This manifold $\body$ is equipped with a \emph{volume form} $\mu$, the \emph{mass measure}~\cite{TN1965}.

A \emph{configuration} of a material medium is represented by an orientation-preserving \emph{embedding}~\cite{Nol1958,Nol1972} (particles cannot occupy the same point in space) of class $C^{\infty}$
\begin{equation*}
  \pp : \body \to \espace, \qquad \XX \mapsto \xx,
\end{equation*}
sometimes referred to as a \emph{placement} in mechanics. The sub-manifold $\Omega_{\pp} = \pp(\body)$ of $\espace$
corresponds to a \emph{configuration system}. The linear tangent map $T\pp: T\body \to T\espace$ will be denoted by $\bF$.

To each embedding $\pp$ corresponds a volume form $\pp_{*}\mu$ on the configuration $\Omega_{\pp} = \pp(\body)$, which is necessarily proportional to the Riemannian volume form $\vol_{\bq}$ on $\Omega_{\pp}$. We have thus the equality
\begin{equation*}
  \pp_{*}\mu = \rho \, \vol_{\bq},
\end{equation*}
which allows to define the \emph{mass density} $\rho$ as a scalar function on the space domain $\Omega_{\pp}$.

\begin{rem}
  It is common to fix a \emph{reference configuration}
  \begin{equation*}
    \Omega_{0} = \pp_{0}(\body),
  \end{equation*}
  and thus to substitute the body $\body$ by a \emph{sub-manifold $\Omega_{0}$ of $\espace$}. In that case, one uses
  \begin{equation*}
    \varphi := \pp \circ \pp_{0}^{-1}, \qquad \Omega_{0} \to \Omega_{\pp},
  \end{equation*}
  called the \emph{deformation}~\cite{Eri1962,GZ1968,TN1965,WT1973,Lub1990,BCCF2001,Hau2002}, rather than $\pp$. Its linear tangent map $T\varphi: T\Omega_{0} \to T\Omega$ is traditionally referred to as the \emph{deformation gradient} and will be denoted by $\bF_{\varphi}$. The mass density $\rho_{0}$ on $\Omega_{0}$, defined by $(\pp_{0})_{*}\mu = \rho_{0} \, \vol_{\bq}$, is then related to $\rho$ by $\rho_{0} = (\rho \circ \varphi) \det \bF_{\varphi}$.
\end{rem}

The configuration space in Continuum Mechanics is thus the set, noted $\Emb(\body,\espace)$, of \emph{smooth embeddings} of $\body$ into $\espace$. This set can be endowed with a \emph{differential manifold structure of infinite dimension}, indeed, an open set of the \emph{Fréchet vector space} $\Cinf(\body,\espace)$~\cite{Mil1984}.

The tangent space to $\Emb(\body,\espace)$ at a point $\pp \in \Emb(\body,\espace)$ is described as follows. Let $\pp(t)$ be a curve in $\Emb(\body,\espace)$ such that $\pp(0) = \pp$, then $(\partial_{t}p)(0) = \VV$ is defined as a tangent vector at $\pp$. The tangent space at $\pp \in \Emb(\body,\espace)$ is thus the vector space
\begin{equation*}
  T_{\pp}\Emb(\body,\espace) = \set{\VV \in \Cinf(\body,T\espace);\; \pi \circ \VV = \pp},
\end{equation*}
where $\VV$ is described by the following diagram:
\begin{equation*}
  \xymatrix{
    T\body \ar[r]^{T\pp} \ar[d]_{\pi}   & T\espace \ar[d]^{\pi} \\
    \body \ar[ur]^{\VV} \ar[r]^{\pp}    & \espace }
\end{equation*}
We recognize $\VV$ as a the \emph{Lagrangian velocity} or a virtual displacement. $T\Emb(\body,\espace)$ is thus the set of \emph{virtual displacements}~\cite{ES1980}.

\begin{rem}
  Since $\Emb(\body,\espace)$ is an open set of a vector space (it is described by only one chart), its tangent bundle is trivial. Indeed, we have
  \begin{equation*}
    T\Emb(\body,\espace) = \Emb(\body,\espace) \times \Cinf(\body,\espace).
  \end{equation*}
\end{rem}

When we specify a path of embeddings $\pp(t)$, then the Lagrangian velocity at time $t$, $\VV(t) = \partial_{t} \pp (t) = \pp_{t}(t)$ belongs to the tangent space $T_{\pp(t)}\Emb(\body,\espace)$. Be careful, $\VV(t)$ \emph{is not, strictly speaking, a vector field}, neither on $\body$, nor on $\Omega_{\pp} = \pp(\body)$! However, vector fields can be constructed from a curve $\pp(t,\XX)$ and its Lagrangian velocity $\pp_{t}(t,\XX)$.
\begin{itemize}
  \item On $\body$, by setting $\UU(t,\XX) := (T\pp^{-1}. \VV)(t,\XX)$;
  \item On $\Omega_{\pp(t)} = \pp(t)(\body)$, by setting $\uu(t,\xx) := (\VV \circ \pp^{-1})(t,\xx)$.
\end{itemize}

These vector fields, respectively called \emph{left Eulerian velocity} and \emph{right Eulerian velocity} play fundamental roles in mechanics. The left Eulerian velocity (defined on the body) is involved in the Eulerian equations of the dynamics of the rigid body~\cite{Eul1765}, while the right Eulerian velocity (defined on space) is involved in fluid mechanics (and also in mechanics of deformable solids). A unified view of these concepts has been proposed by Arnold in an article dedicated to the bicentennial of the Euler equations of the rigid body~\cite{Arn1966}.

To each embedding $\pp \in \Emb(\body,\espace)$ corresponds a Riemannian metric
\begin{equation*}
  \bgamma = \pp^{*}\bq \in \Met(\body),
\end{equation*}
where $\Met(\body)$ designates the set of all the Riemannian metrics defined on $\body$. This mapping is however not injective, indeed, $ \bar{\pp} = g \circ \pp$, where $g \in \Isom(\espace,\bq)$ induces the same metric $ \bgamma = \bar{\pp}^{*}\bq$ on $\body$. Note that the Riemannian curvature of $\bgamma$ vanishes when $\body$ is of dimension 3 (this is obviously no longer true in shell theory where $\body$ is a manifold of dimension 2).

\section{Strains and stresses}
\label{sec:strain-and-stress}

A motion in Continuum Mechanics corresponds to a curve $\pp(t)$ in $\Emb(\body,\espace)$~\cite{Nol1958,Nol1972} (a \emph{path of embeddings}). To this motion is associated a Lagrangian velocity
\begin{equation*}
  \partial_{t}p(t,\XX) = \VV(t,\XX),
\end{equation*}
together with right and left Eulerian velocities
\begin{equation*}
  \uu(t,\xx) = \VV(t,\pp^{-1}(t,\xx)), \qquad \UU(t,\XX) = (T_{\XX}\pp)^{-1}.\VV(t,\XX),
\end{equation*}
and a \emph{strain rate}, a second-order symmetric covariant tensor field on $\Omega_{\pp}$
\begin{equation}\label{eq:strain-rate}
  \bd := \frac{1}{2} \Lie_{\uu} \bq ,
\end{equation}
where $\Lie_{\uu}$ is the Lie derivative with respect to the (right) Eulerian velocity $\uu$.

\begin{rem}
  Traditionally, the strain rate is introduced using its mixed form $\widehat{\bd} := \bq^{-1} \bd$, which writes as
  \begin{equation*}
    \widehat{\bd} = \frac{1}{2} \left( \nabla \uu + (\nabla \uu)^{t}\right),
  \end{equation*}
  where $\nabla \uu$ is the covariant derivative of the Eulerian velocity $\uu$ and $(\nabla \uu)^{t}$ is the transpose (relative to the metric $\bq$, see \autoref{sec:second-order-tensors}) of the linear operator $\ww \mapsto \nabla_{\ww} \uu$. The connection between these two expressions results from the formula
  \begin{equation*}
    2 \bd = \Lie_{\uu} \bq = \bq \nabla \uu + \bq(\nabla \uu)^{t} = 2\bq \widehat{\bd},
  \end{equation*}
  which can be deduced from proposition~\ref{prop:Lie-derivative-versus-covariant-derivative}. We have furthermore
  \begin{equation*}
    \bd = D \uu^\flat,
    \qquad
    \uu^\flat=\bq \uu,
  \end{equation*}
  where the operator
  \begin{equation}\label{eq:DZ}
    D\alpha  (X,Y) = \frac{1}{2} \left( (\nabla_X \alpha(Y) + (\nabla_Y \alpha)(X)\right)
  \end{equation}
  is the \emph{formal adjoint} of the divergence of a second-order symmetric tensor field.
\end{rem}

The following result, where $\bgamma_{t}=\partial_{t} \bgamma$, is a direct consequence of theorem~\ref{lem:fundamental-time-derivative-formula}.

\begin{thm}[Rougée, 1980, 1991]\label{theo:metric-strain-rate}
  Along a deformation path $\pp(t)$ of embeddings in $\Emb(\body,\espace)$, the Riemannian metric over $\body$, $\bgamma(t) = \pp(t)^{*}\bq$, satisfies the evolution equation
  \begin{equation*}
    \bgamma_{t} = 2 \pp^{*}\bd.
  \end{equation*}
\end{thm}

\begin{rem}\label{rem:D-body}
  Setting $\bD := \pp^{*}\bd$ and $\bgamma:= \pp^{*}\bq$, we have
  \begin{equation*}
    \bD = \frac{1}{2} \bgamma\left(\nabla^{\bgamma} \UU + (\nabla^{\bgamma} \UU\right)^{t}) = D^\bgamma \UU^\flat ,
  \end{equation*}
  where $\UU$ is the left Eulerian velocity, $\UU^\flat=\bgamma \UU$, and where $D^\bgamma$ is the formal adjoint of the divergence operator relative to the metric $\bgamma$.
\end{rem}

Given a reference configuration $\pp_{0}: \body \to \Omega_{0}$, the \emph{deformation} $\varphi$ writes $\varphi = \pp\circ \pp_{0}^{-1}$. Then, the \emph{right Cauchy--Green tensor} is defined on the reference configuration $\Omega_{0} = p_{0}(\body)$ by
\begin{equation}\label{eq:right-Cauchy-Green-tensor}
  \bC := \varphi^{*}\bq=\bF_{\varphi}^{\star}\bq \,\bF_{\varphi} = \bq \bF_{\varphi}^{t} \bF_{\varphi},
\end{equation}
and the \emph{left Cauchy--Green tensor}, on the deformed configuration $\Omega_{p} = p(\body)$, by
\begin{equation}\label{eq:left-Cauchy-Green-tensor}
  \bb := \varphi_{*}\bq^{-1} = \bF_{\varphi}\bq^{-1}\bF_{\varphi}^{\star} = \bF_{\varphi} \bF_{\varphi}^{t} \bq^{-1}.
\end{equation}

\begin{rem}\label{rem:CG-pullbacks}
  The Cauchy--Green tensors are related to the metrics $\bgamma$ and $\bgamma_{0}$ by
  \begin{equation*}
    \pp_{0}^{*}\bC = \pp_{0}^{*}\varphi^{*}\bq = \pp^{*}\bq = \bgamma ,
  \end{equation*}
  and
  \begin{equation*}
    p^{*}\bb = p^{*}\varphi_{*}\bq^{-1} = \pp_{0}^{*}\bq^{-1} = \bgamma_{0}^{-1} .
  \end{equation*}
\end{rem}

The notion of stress is dual to that of deformation. The most common concept is the one of \emph{Cauchy stress tensor} $\bsigma$, defined on the configuration $\Omega_{\pp} = \pp(\body)$ (a symmetric contravariant tensor field). To this dual concept of deformations is associated a linear functional (which corresponds to the opposite of the \emph{virtual work of internal forces})
\begin{equation*}
  \mathcal{P}(\ww) = \int_{\Omega} (\bsigma : D \ww^\flat) \, \vol_{\bq} = \int_{\Omega} (\btau : D \ww^\flat )\, \rho \, \vol_{\bq},
  \qquad
  D \ww^\flat=\frac{1}{2} \Lie_{\ww} \bq,
\end{equation*}
where the double-dots means the double contraction ($\ba :\bb = a^{ij}b_{ij}$), $\ww\in T\Omega$ is an Eulerian virtual velocity, $\ww^\flat=\bq \ww$, and where $\btau := \bsigma /\rho$ is the \emph{Kirchhoff stress tensor}. A geometrical interpretation of the stresses in terms of virtual work can be found in~\cite{ES1980,SR1999}.

This leads us to introduce two \emph{stress tensors on the body}, the \emph{Noll stress tensor}~\cite{Nol1972,Nol1978}
\begin{equation*}
  \msigma := \pp^{*}\bsigma ,
\end{equation*}
and the \emph{Rougée stress tensor}~\cite{Rou1980,Rou1991}
\begin{equation}\label{eq:theta}
  \btheta := \pp^{*}\btau ,
\end{equation}
which are both contravariant and symmetric, and appear naturally when writing down the push-forward of the virtual work of internal forces on the body
\begin{equation*}
  \mathcal{P}(\WW) = \int_{\body} (\msigma : D^\bgamma \WW^\flat ) \, \vol_{\bgamma} = \int_{\body} (\btheta : D^\bgamma \WW^\flat ) \, \mu,
\end{equation*}
where $\WW=\pp^{*} \ww$ is a left Eulerian virtual velocity, $\WW^\flat=\bgamma \WW$, and the operator $D^\bgamma$ has been defined in remark~\ref{rem:D-body}.

\begin{rem}
  When one works with a reference configuration $\Omega_{0}$, rather than $\body$, and if we introduce the \emph{deformation} $\varphi = \pp\circ \pp_{0}^{-1}$, then, the \emph{second Piola--Kirchhoff tensor} $\bS$, symmetric and contravariant, is both the pull-back of Kirchhoff stress by $\varphi$ and the push-forward of Rougée's stress by $\pp_{0}$,
  \begin{equation*}
    \bS := \varphi^{*}\btau =\pp_{0*} \btheta
  \end{equation*}
  It is defined on the reference configuration $\Omega_{0}$.
\end{rem}

\section{The manifold of Riemannian metrics}
\label{sec:manifold-of-metrics}

The set $\Met(\body)$ of all Riemannian metrics defined on $\body$ thus acquires some importance in Continuum Mechanics. The reader may refer to the work of Rougée~\cite{Rou1997,Rou2006} (see also~\cite{KD2019,KD2021}) for a geometrical formulation of hyper-elasticity such as a vector field on $\Met(\body)$, in other words, as a section $F: \Met(\body) \to T\Met(\body)$ of the tangent vector bundle $T\Met(\body)$. It turns out that $\Met(\body)$ is itself a Fréchet manifold. It is in fact an \emph{open convex set} of the infinite dimensional vector space (a Fréchet space but not a Banach space)
\begin{equation*}
  \Gamma(S^{2}T^{\star}\body),
\end{equation*}
of smooth sections of the vector bundle of covariant symmetric tensors of order two. The tangent space $T_{\bgamma}\Met(\body)$ is canonically identified with the vector space $\Gamma(S^{2}T^{\star}\body)$ and the tangent vector bundle $T\Met(\body)$ is trivial,
\begin{equation*}
  T\Met(\body) = \Met(\body) \times \Gamma(S^{2}T^{\star}\body).
\end{equation*}

The cotangent bundle $T^{\star}\Met(\body)$ of $\Met(\body)$ writes
\begin{equation*}
  T^{\star}\Met(\body) = \Met(\body) \times \Gamma(S^{2}T^{\star}\body)^{\star},
\end{equation*}
where $(\Gamma(S^{2}T^{\star}\body))^{\star}$ is the space of \emph{tensor-distributions}, \textit{i.e.} continuous linear functionals on $\Gamma(S^{2}T^{\star}\body)$ (with compact support).

\begin{rem}
  The general concept of tensor-distributions seems to have been introduced by Lichnerowicz~\cite{Lic1994} and extends \emph{Schwartz distributions} from functions to tensor fields. When these tensor fields are covariant and alternate (\textit{i.e.} differential forms) they are called \emph{de Rham's currents}. Surprisingly, currents (whose origin comes from physics) have been discovered before Schwartz's distributions and have inspired him to formulate his theory~\cite{Luet1982}.
\end{rem}

From the mechanical point of view, an element $\mathcal{P}_{\bgamma}$ of $T_{\bgamma}^{\star}\Met(\body)$ can be interpreted as the virtual work of internal forces. Among this large space of tensor-distributions, there is an important subset of them which write as
\begin{equation*}
  \mathcal{P}_{\bgamma}(\bepsilon) = \int_{\body} (\btheta : \bepsilon) \mu,
\end{equation*}
where $\btheta$ is a contravariant second-order tensor field on $\body$ and the double-dots means the double contraction. These particular distributions will be called \emph{distributions with density}.

\begin{rem}
  The correct setting for studying these infinite dimensional manifolds of smooth mappings~\cite{Mic1980} seems to be what is now called \emph{convenient calculus}~\cite{KM1997,Bru2018}, and has been extensively studied these last four decades. Thanks to the concept of \emph{Frölicher spaces}~\cite{KM1997} (a more restrictive concept than \emph{Diffeology}~\cite{Igl2013}, but more adapted to our needs \cite{SE2020}), the definition of \emph{differentiability} on manifolds of smooth mappings has become, nowadays, less tricky. For instance, it is now possible to overcome some difficulties which were inherent in the theory of Fréchet spaces. This last category is not nice from the point of view of analysis, because in general the topological dual of a Fréchet vector space (like $T_{\bgamma}^{\star}\Met(\body)$) or the space of continuous linear mappings between two Fréchet vector spaces are not Fréchet spaces~\cite{Ham1982}. They are, however, still Frölicher spaces~\cite{FK1988,KM1997}.
\end{rem}

\subsection{Weak Riemannian structure on $\Met(\body)$}
\label{subsec:weak-Riemannian-structure}

The manifold $\Met(\body)$ of Riemannian metrics can be equipped itself with a natural Riemannian structure, by setting
\begin{equation}\label{eq:Rougee-metric}
  G^{\mu}_{\bgamma}(\bepsilon^{1}, \bepsilon^{2}) := \int_{\body} \tr(\bgamma^{-1}\bepsilon^{1} \bgamma^{-1}\bepsilon^{2}) \, \mu, \qquad \bepsilon^{1}, \bepsilon^{2} \in T_{\bgamma}\Met(\body),
\end{equation}
where $\tr(\bgamma^{-1}\bepsilon^{1} \bgamma^{-1}\bepsilon^{2}) = \gamma^{ij}\varepsilon^{1}_{ik} \gamma^{kl}\varepsilon^{2}_{jl}$, in a local coordinate system. This (meta-)metric~\eqref{eq:Rougee-metric} was introduced by Rougée~\cite{Rou1997,Rou2006} and seems well adapted for the geometrical formulation of Cauchy elasticity in finite strains. For this reason, we will call it the \emph{Rougée metric}.

\begin{rem}
  In~\eqref{eq:Rougee-metric}, a point $\XX \in \body$ being fixed, $\tr(\bgamma^{-1}(\XX)\bepsilon^{1}(\XX) \bgamma^{-1}(\XX)\bepsilon^{2}(\XX))$ corresponds to a scalar product on the finite dimensional space of symmetric matrices and \eqref{eq:Rougee-metric} can be interpreted as the ``mean value'' of these scalar products. In several papers in solid mechanics (including Rougée's papers) the integral is omitted, and as noted by Fiala~\cite{Fia2004}, this lead to some confusion regarding the problem of on which space the metric is defined. From the mathematical point of view, the geometry associated with this metric is nevertheless somehow \emph{pointwise}, as emphasized in~\cite{Bru2018}. The reason of this confusion (which is present in Rougée's paper~\cite{Rou2006}) is the fact that, besides more recent works on the manifold of Riemannian metrics, there is also an important literature on the geometry of positive definite matrices (see for instance the book~\cite{Bha2009}) and it is easy to confuse both concepts. The metrics discussed in our paper are obtained by a straightforward integration of the corresponding metric on the (finite dimensional) manifold of positive definite matrices.
\end{rem}

In general relativity, where no volume form is defined \textit{a priori}, one uses a variant of this metric, the \emph{Ebin metric}~\cite{Ebi1968,FG1989,Cla2010} defined as
\begin{equation}\label{eq:Ebin-metric}
  G_{\bgamma}(\bepsilon^{1}, \bepsilon^{2}) := \int_{\body} \tr(\bgamma^{-1}\bepsilon^{1} \bgamma^{-1}\bepsilon^{2}) \, \vol_{\bgamma}, \qquad \bepsilon^{1}, \bepsilon^{2} \in T_{\bgamma}\Met(\body).
\end{equation}
where $\vol_{\bgamma}$ is the Riemannian volume associated with the metric $\bgamma$.

\begin{rem}
  An important property of the Ebin metric is that it is invariant by the diffeomorphism group $\Diff(\body)$. More specifically:
  \begin{equation*}
    G_{\varphi^{*}\bgamma}(\varphi^{*}\bepsilon^{1}, \varphi^{*}\bepsilon^{2}) = G_{\bgamma}(\bepsilon^{1}, \bepsilon^{2}), \qquad \forall \varphi \in \Diff(\body).
  \end{equation*}
  Contrary to the Ebin metric $G$, the Rougée metric $G^{\mu}$ is not invariant by the diffeomorphism group $\Diff(\body)$ but only by its subgroup
  \begin{equation*}
    \Diff_{\mu}(\body) := \set{\varphi \in \Diff(\body); \; \varphi^{*}\mu = \mu},
  \end{equation*}
  of diffeomorphisms which preserve the volume form $\mu$ (the mass measure).
\end{rem}

These Riemannian structures on $\Met(\body)$ are relatively well understood and have been intensively studied~\cite{Ebi1967,Ebi1968,FG1989,GM1991,Cla2010}. There are, however, important differences between Riemannian geometry on a finite dimensional manifold and on an infinite dimensional manifold. A Riemannian metric $G$ defined on a manifold $M$ (of finite or infinite dimension) induces a mapping
\begin{equation*}
  G : TM \to T^{\star}M
\end{equation*}
which is injective (a metric is, at each point, a non-degenerate symmetric bilinear form). If $M$ is of finite dimension this mapping is necessarily bijective but this is no longer necessarily true in infinite dimension. We will then distinguish a \emph{strong} Riemannian metric (if $G$ is bijective) from a \emph{weak} Riemannian metric (if $G$ is only injective).

The Rougée metric induces a linear injective (but not surjective) mapping
\begin{equation*}
  T_{\bgamma}\Met(\body) \to T_{\bgamma}^{\star}\Met(\body), \qquad \bbeta \mapsto G^{\mu}_{\bgamma}(\bbeta,\cdot),
\end{equation*}
whose range corresponds to \emph{distributions with density}. In other words, an element $\mathcal{P}_{\bgamma}$ belongs to this range if it writes
\begin{equation*}
  \mathcal{P}_{\bgamma}(\bepsilon) = \int_{\body} (\btheta : \bepsilon) \mu, \quad \text{where} \quad \btheta = \bgamma^{-1} \bbeta \bgamma^{-1},
\end{equation*}
for some $\bbeta \in T_{\bgamma}\Met(\body)$, defining on the body the \emph{Rougée stress tensor} $\btheta$, according to \eqref{eq:theta}, as the density of the distribution $\mathcal{P}_{\bgamma}$. An elasticity law (in the Cauchy sense) writes thus as
\begin{equation}\label{eq:elastic-law-in-the-body}
  \btheta = \bgamma^{-1} F(\bgamma) \bgamma^{-1},
\end{equation}
where $F$ is a vector field on $\Met(\body)$. This formula is better understood using the following diagram
\begin{equation*}
  \xymatrix{
  T\Met(\body) \ar[r]^{G^{\mu}} \ar[d]^{\pi} & T^{\star}\Met(\body) \\
  \Met(\body) \ar@/^1pc/[u]^{F} \ar[ur]_{\quad \btheta = \bgamma^{-1} F(\bgamma) \bgamma^{-1}} &  }
\end{equation*}
By \emph{push-forward} on $\Omega=\pp(\body)$, we recover the Cauchy stress tensor
\begin{equation}\label{eq:elastic-law-in-space}
  \bsigma = \rho \, \pp_{*}\btheta = \rho \, \bq^{-1} \pp_{*}(F(\pp^{*}\bq))\bq^{-1}.
\end{equation}

\subsection{Riemannian product structure on $\Met(\body)$}
\label{subsec:Riemannian-product-structure}

Finally, we will describe a natural \emph{Riemannian product structure} on $(\Met(\body), G^{\mu})$ which has a strong mechanical signification. Note first that $\Met(\body)$ being an open set of $\Gamma(S^{2}T^{\star}\body)$, an element $\bgamma \in \Met(\body)$ can also be seen as an element of $T_{\bgamma}\Met(\body)$ and we can thus consider the mapping
\begin{equation*}
  \bgamma \mapsto \bgamma, \qquad \Met(\body) \to T\Met(\body)
\end{equation*}
as a section of the tangent vector bundle, \textit{i.e.} a vector field on $\Met(\body)$. Therefore, we can write
\begin{equation}\label{eq:dec-tangent-space}
  T_{\bgamma}\Met(\body) = \Cinf(\body) \bgamma \oplus (\Cinf(\body) \bgamma)^{\bot},
\end{equation}
where orthogonality refers either to Ebin's metric~\eqref{eq:Ebin-metric}, or to Rougée's metric~\eqref{eq:Rougee-metric}. In both cases, the space $(\Cinf(\body) \bgamma)^{\bot}$ writes
\begin{equation*}
  (\Cinf(\body) \bgamma)^{\bot} = \set{\bepsilon \in \Gamma(S^{2}T^{\star}\body);\; \tr(\bgamma^{-1}\bepsilon) = 0},
\end{equation*}
and we have the following result.

\begin{lem}
  Let $\pp \in \Emb(\body,\espace)$, $\bk$ a second-order covariant tensor field on $\Omega_{\pp} = \pp(\body)$ and
  \begin{equation*}
    \bk =  \bk^{H} + \bk^{D},
  \end{equation*}
  its decomposition into a \emph{spherical part} (or hydrostatic part) and a {deviator} (with respect to the metric $\bq$ on $\espace$). So, the pull-back of this decomposition
  \begin{equation*}
    \pp^{*}\bk =  \pp^{*}\bk^{H} + \pp^{*}\bk^{D}
  \end{equation*}
  corresponds to the orthogonal decomposition
  \begin{equation*}
    T_{\bgamma}\Met(\body) = \Cinf(\body) \bgamma \oplus (\Cinf(\body) \bgamma)^{\bot}.
  \end{equation*}
\end{lem}

To this decomposition of the tangent bundle~\eqref{eq:dec-tangent-space} corresponds a Riemannian manifold product structure of $(\Met(\body),G^{\mu})$. Indeed, let $\Vol(\body)$ be the set of volume forms on $\body$, which is an open positive cone in $\Omega^{3}(\body)$. Given a volume form $\mu$ on $\body$, we define the sub-manifold
\begin{equation*}
  \Met_{\mu}(\body) := \set{\bgamma \in \Met(\body);\; \vol_{\bgamma} = \mu}
\end{equation*}
of $\Met(\body)$ and the mapping
\begin{equation*}
  \Psi_{\mu} : \Vol(\body) \times \Met_{\mu}(\body) \to \Met(\body), \qquad (\nu, \bgamma) \mapsto (\nu/\mu)^{2/3}\bgamma.
\end{equation*}
Then, the mapping $\Psi_{\mu}$ is a Riemannian isometry if we endow $\Met(\body)$ and its submanifold $\Met_{\mu}(\body)$ with the Rougée metric $G^{\mu}$, and the manifold $\Vol(\body)$ with the following Riemannian metric
\begin{equation*}
  G^{\mu}_{\nu}(\omega_{1}, \omega_{2}) := \frac{4}{3}\int_{\body} \left(\frac{\omega_{1}}{\nu}\right)\left(\frac{\omega_{2}}{\nu}\right)\, \mu.
\end{equation*}

\begin{rem}
  A similar decomposition exists for the Riemannian manifold $(\Met(\body),G)$, where $G$ is the Ebin metric~\cite{Cla2010}, but then, one needs to endow $\Vol(\body)$ with the following Riemannian metric
  \begin{equation*}
    G_{\nu}(\omega_{1}, \omega_{2}) := \frac{4}{3}\int_{\body} \left(\frac{\omega_{1}}{\nu}\right)\left(\frac{\omega_{2}}{\nu}\right)\, \nu.
  \end{equation*}
\end{rem}

\subsection{Covariant derivatives on $T\Met(\body)$}
\label{subsec:covariant-derivative-on-TMet}

On a Fréchet vector space, only the notion of directional derivative (or derivative along a curve) makes sense. On $\Gamma(S^{2}T^{\star}\body)$, this one is written
\begin{equation}\label{eq:canonical-covariant-derivative}
  \partial_{t} \bepsilon := \partial_{t}\bepsilon(t,\XX),
\end{equation}
where $\bepsilon$ is a time-dependent tensor field and the partial derivative with respect to $t$ occurs in each (finite dimensional) vector space $T_{\XX}^{\star}\body$.

More generally, a covariant derivative on $T\Met(\body)$ is a linear operator $D$ that associates to each vector field $\bepsilon(t)$ defined along a curve $\bgamma(t) \in \Met(\body)$ (\textit{i.e.}, $\pi \circ \bepsilon(t) = \bgamma(t)$), a vector field $D_{t}\bepsilon$ defined along $\bgamma(t)$ and which satisfies the Leibniz rule
\begin{equation*}
  D_{t}(f\bepsilon) = (\partial_{t}f) \bepsilon + f D_{t}(\bepsilon),
\end{equation*}
for any numeric function $f:t \mapsto f(t)$. Of course, without further assumption, this definition is tricky. However, in this paper, and more generally when dealing with Riemannian geometry on infinite dimensional manifolds, one usually restricts to \emph{local operators}. This means that the considered definition involves only the pointwise value of the derivatives of the involved fields~\cite{Bru2018}, up to a given order. In other words it is assumed that these operators depend only on a finite number of \emph{jets} or \emph{gradients} (as it is common to call them in mechanics) of these fields.

In particular, $\partial_{t}$ defines a covariant derivative on the vector space $\Gamma(S^{2}T^{\star}\body)$ which, by restriction, induces on the open set $\Met(\body)$ a covariant derivative that can be considered as canonical. Any other covariant derivative on $T\Met(\body)$ can therefore be written as
\begin{equation*}
  D_{t}\bepsilon = \partial_{t}\bepsilon + \Gamma_{\bgamma}(\bgamma_{t},\bepsilon),
  \qquad
  \bgamma_{t}=\partial_{t}\bgamma,
\end{equation*}
where
\begin{equation*}
  \Gamma_{\bgamma}: T_{\bgamma}\Met(\body) \times T_{\bgamma}\Met(\body) \to T_{\bgamma}\Met(\body)
\end{equation*}
is a continuous bilinear operator called the \emph{Christoffel operator}. It is the analogous, in infinite dimension, of the Christoffel symbols $\Gamma_{ij}^{k}$ in finite dimension. This formula is true in finite dimension, but could be false in infinite dimension even if apparently no counter-example is known~\cite{Bru2018}. We will restrict the definition of covariant derivatives to those which can be written this way, where $\Gamma$ is a local operator.

The Riemannian covariant derivative associated with a metric $G$ on $\Met(\body)$ is characterized, on one hand, by being \emph{compatible with the metric $G$ }, i.e.
\begin{equation*}
  \dd{}{t}G_{\bgamma(t)}(\bepsilon^{1}(t), \bepsilon^{2}(t)) = G_{\bgamma(t)}(D_{t}\bepsilon^{1}(t), \bepsilon^{2}(t)) + G_{\bgamma(t)}(\bepsilon^{1}(t), D_{t}\bepsilon^{2}(t)),
\end{equation*}
for any one-parameter family $\bgamma(t)\in \Met(\body)$ and all vector fields $\bepsilon^{1}(t), \bepsilon^{2}(t)$ defined along $\bgamma(t)$ and, on the other hand, by the fact that it is \emph{symmetric}, i.e.
\begin{equation*}
  D_{s}\partial_{t}\bgamma(t,s) = D_{t}\partial_{s}\bgamma(t,s).
\end{equation*}
for any two-parameters family $\bgamma(t,s) \in \Met(\body)$, meaning here that $\Gamma_{\bgamma}(\bepsilon^{2},\bepsilon^{1}) = \Gamma_{\bgamma}(\bepsilon^{1},\bepsilon^{2})$. Note, however, that for a weak Riemannian metric, as it is the case with the Riemannian structure on $\Met(\body)$, only the uniqueness of a symmetric covariant derivative that preserves the metric (Levi-Civita connection) is ensured, but not its existence, \textit{a priori}.

\begin{thm}
  The Rougée metric $G^{\mu}$ defined by~\eqref{eq:Rougee-metric} over $\Met(\body)$ admits the following unique symmetric covariant derivative compatible with it
  \begin{equation}\label{eq:Rougee-covariant-derivative}
    D_{t}\bepsilon := \partial_{t} \bepsilon - \frac{1}{2} \left(\bgamma_{t}\bgamma^{-1}\bepsilon + \bepsilon\bgamma^{-1}\bgamma_{t}\right).
  \end{equation}
  Its curvature is non-null and writes
  \begin{equation*}
    R(\partial_s, \partial_{t})\bepsilon = \frac{1}{4}\bgamma\left[[\bgamma^{-1}\bgamma_{t},\bgamma^{-1}\bgamma_s],\bgamma^{-1}\bepsilon\right],
  \end{equation*}
  where the notation $[\cdot, \cdot]$ means the commutator of two mixed tensors of order two.
\end{thm}

The following notable fact can be observed: although the Rougée metric explicitly depends on the volume form $\mu$, the associated covariant derivative does not depend on it. This covariant derivative is moreover invariant by the action of the whole diffeomorphism group of $\Met(\body)$ (whereas the metric is invariant by the diffeomorphisms which preserve the volume form $\mu$). In other words, given a diffeomorphism of the body $\psi \in \Diff(\body)$, a curve $\widetilde{\bgamma}: t \mapsto \bgamma(t)$ on $\Met(\body)$ and a curve $\bepsilon(t)$ on $T\Met(\body)$ defined along $\widetilde{\bgamma}$, and using the notation
$\dcov{\widetilde{\bgamma}}{}$ rather than $D_{t}$, to insist on the dependence on the path $\widetilde{\bgamma}$, this means that
\begin{equation*}
  \dcov{\psi^{*}\widetilde{\bgamma}}{(\psi^{*}\bepsilon)} = \psi^{*}\left(\dcov{\widetilde{\bgamma}}{\bepsilon}\right)
\end{equation*}
and thus
\begin{equation*}
  \Gamma_{\psi^{*}\bgamma}(\psi^{*}\bgamma_{t},\psi^{*}\bepsilon) = \psi^{*}\left(\Gamma_{\bgamma}(\bgamma_{t},\bepsilon)\right).
\end{equation*}

\begin{rem}
  The covariant derivative associated with the Ebin metric~\eqref{eq:Ebin-metric} on $\Met(\body)$ has been calculated in~\cite{GM1991}. It is slightly more complicated and writes
  \begin{equation}\label{eq:Ebin-covariant-derivative}
    D_{t}\bepsilon := \partial_{t} \bepsilon - \frac{1}{2} \left( \bgamma_{t}\bgamma^{-1}\bepsilon + \bepsilon\bgamma^{-1}\bgamma_{t} + \frac{1}{2} \tr(\bgamma^{-1}\bgamma_{t}\bgamma^{-1}\bepsilon)\bgamma - \frac{1}{2}\tr(\bgamma^{-1}\bgamma_{t})\bepsilon - \frac{1}{2}\tr(\bgamma^{-1}\bepsilon)\bgamma_{t} \right).
  \end{equation}
  It is also invariant by $\Diff(\body)$.
\end{rem}

\subsection{Geodesics and the Riemannian exponential mapping}
\label{subsec:geodesics}

In Riemannian geometry, geodesics are defined as the extremals of the \emph{energy functional}
\begin{equation*}
  E(\bgamma) := \frac{1}{2} \int_{0}^{1} \langle \bgamma_{t}, \bgamma_{t} \rangle \, dt.
\end{equation*}
For the Rougée metric, they are therefore solutions of the associated Euler-Lagrange equation
\begin{equation*}
  D_{t}\bgamma_{t} = \bgamma_{tt} - \bgamma_{t}\bgamma^{-1}\bgamma_{t} = 0,
\end{equation*}
where $\bgamma_{tt}=\partial_{t}^{2}\bgamma$. The crucial observation is that the resolution of this equation is pointwise in nature, partial derivatives (relative to $t$) are to be understood as derivatives in the finite dimensional vector space $S^{2}T_{\XX}^{\star}\body$, where $\XX$ is fixed. This second-order differential equation recast as
\begin{equation*}
  \partial_{t}(\bgamma^{-1}\bgamma_{t}) = 0.
\end{equation*}
Given initial values $(\bgamma_{0},\bepsilon_{0}) \in \Met(\body) \times \Gamma(S^{2}T^{\star}\body)$, and introducing the mixed tensor $\widehat{\bepsilon}_{0} := \bgamma_{0}^{-1}\bepsilon_{0}$, we, obtain thus
\begin{equation*}
  \bgamma(t) = \bgamma_{0}\exp\left(t\widehat{\bepsilon}_{0}\right),
\end{equation*}
where here, $\exp(t\widehat{\bepsilon}_{0})(\XX)$ means the exponential of the linear endomorphism $t\widehat{\bepsilon}_{0}(\XX)$ of the finite dimensional vector space $T_{\XX}\body$.

We deduce then the expression of the \emph{Riemannian exponential mapping} associated with the Rougée metric $G^{\mu}$. It is defined as the time $1$ of the geodesic flow $\Phi(t,\bgamma_{0},\bepsilon)$, where $(\bgamma_{0},\bepsilon)$ are the initial data (position--velocity) at time $t=0$ of the geodesic $\bgamma(t) = \Phi(t,\bgamma_{0},\bepsilon)$. So we have
\begin{equation*}
  \Exp_{\bgamma_{0}}(\bepsilon) = \bgamma_{0}\exp\left(\bgamma_{0}^{-1}\bepsilon\right).
\end{equation*}

This exponential mapping is a global diffeomorphism from $T_{\bgamma_{0}}\Met(\body)$ onto $\Met(\body)$ for any $\bgamma_{0}\in \Met(\body)$ and provides a \emph{global chart} for $\Met(\body)$ by second-order symmetric covariant tensor fields. The proof is similar to that of~\cite[Proposition 2]{Cla2010} (see also~\cite{FG1989,GM1991}). It results essentially from the fact that, given a finite dimensional Euclidean space $(V,\bq)$, the exponential mapping (understood as the exponential of an endomorphism of a finite dimensional vector space)
\begin{equation*}
  \exp : S^{2}V \to S_{+}^{2}V
\end{equation*}
is a global diffeomorphism~\cite{Sou1997}.

\begin{rem}
  The logarithm of the metric $\bC=\varphi^{*}\bq$ on a reference configuration $\Omega_{0}$ (right Cauchy--Green tensor) has been initially introduced by Becker~\cite{Bec1893} and Hencky~\cite{Hen1928} (see also~\cite{MMEN2018}) to define the \emph{true strain tensor},
  \begin{equation*}
    \bE = \frac{1}{2} \bq \ln (\bq^{-1} \bC),
  \end{equation*}
  on the reference configuration $\Omega_{0}$. It was later successfully used to formulate finite strain elasto-plasticity of metallic materials in~\cite{MAL2002}, where the authors have extended to finite strains the additive decomposition of the total strain $\bE$ into an elastic strain $\bE^{e}$ and a traceless plastic strain $\bE^{p}$. This extended decomposition recasts, in this more general geometric framework\footnote{This formulation does not requires the introduction of the decomposition $RU$.} on $\Met(\body)$ introduced in~\cite{Rou1991a}, thanks to the Riemannian logarithm $\Log_{\bgamma_{0}}$ (the inverse of the Riemannian exponential mapping
  $\Exp_{\bgamma_{0}}$), as
  \begin{equation*}
    \pmb{\mathfrak E}=\pp_{0}^{\, *} \bE=\frac{1}{2} \Log_{\bgamma_{0}} \bgamma,
  \end{equation*}
  and thus
  \begin{equation*}
    \pmb{\mathfrak E}= \pmb{\mathfrak E}^{e}+ \pmb{\mathfrak E}^{p},
    \qquad
    \begin{cases}
      \pmb{\mathfrak E}^{e}=\pp_{0}^{\, *} \bE^{e} ,
      \\
      \pmb{\mathfrak E}^{p}=\pp_{0}^{\, *} \bE^{p},
    \end{cases}
  \end{equation*}
  where $\tr(\bgamma_{0}^{-1} \pmb{\mathfrak E}^{p})= \pp_{0}^{\, *} \tr(\bq^{-1} \bE^{p})=0$.
\end{rem}

\subsection{Covariant derivatives on $T^{\star}\Met(\body)$}
\label{subsec:covariant-derivative-on-TstarMet}

Recall that the cotangent vector space $T_{\bgamma}^{\star}\Met(\body)$ is a space of \emph{tensor-distributions}, \emph{i.e.} continuous linear functionals over the space of the virtual deformation fields $T_{\bgamma}\Met(\body)$. Any covariant derivative on $T\Met(\body)$ formally induces a covariant derivative on the cotangent bundle $T^{\star}\Met(\body)$, thanks to the \emph{Leibniz rule}, defined by
\begin{equation}\label{eq:nabla-P}
  (D_{t}\mathcal{P})(\bepsilon) = \partial_{t}(\mathcal{P}(\bepsilon)) - \mathcal{P} \left(D_{t}\bepsilon\right),
\end{equation}
for any \emph{covector field} $\mathcal{P}$ and any vector field $\bepsilon$ defined along a curve $\bgamma(t) \in \Met(\body)$. When $\mathcal{P}$ is a distribution with density, that is when
\begin{equation*}
  \mathcal{P}(\bepsilon) = \int_{\body} (\btheta:\bepsilon) \mu,
\end{equation*}
where $\btheta$ is a symmetric second-order contravariant tensor field on $\body$, we get
\begin{equation*}
  (D_{t}\mathcal{P})(\bepsilon) = \int_{\body} \left( \partial_{t}\btheta : \bepsilon - \btheta : \Gamma_{\bgamma}(\bgamma_{t},\bepsilon) \right) \mu.
\end{equation*}
Note, however, that $D_{t}\mathcal{P}$ may not be a distribution with density, because it is not at all obvious that the expression
\begin{equation*}
  \partial_{t}\btheta : \bepsilon - \btheta : \Gamma_{\bgamma}(\bgamma_{t},\bepsilon)
\end{equation*}
can be recast as the contraction of some contravariant tensor field with $\bepsilon$. If this is the case (for all $\bgamma$, $\btheta$ and $\bepsilon$), we then denote this contravariant tensor field (density) by $D_{t}\btheta$, and write
\begin{equation*}
  D_{t}\mathcal{P}(\bepsilon) = \int_{\body} \left( D_{t}\btheta : \bepsilon \right) \mu .
\end{equation*}
In that case, we say that the covariant derivative $D$ \emph{preserves distributions with densities}, and we have the Leibniz rule
\begin{equation*}
  D_{t}\btheta : \bepsilon + \btheta : D_{t}\bepsilon=\partial_{t}\left(\btheta : \bepsilon\right).
\end{equation*}

\begin{rem}\label{rem:local-covariant-derivatives}
  The covariant derivative $D$ always preserves distributions with densities when $\Gamma_{\bgamma}(\bgamma_{t},\bepsilon)$ depends only on $\bepsilon$ through its $0$-jets, which means that $\Gamma_{\bgamma}(\bgamma_{t},\bepsilon)(\XX)$ depends only on $\bepsilon(\XX)$, and we will write then
  \begin{equation*}
    \Gamma_{\bgamma}(\bgamma_{t},\bepsilon)(\XX) = \Gamma_{\bgamma}(\bgamma_{t},\bepsilon(\XX)).
  \end{equation*}
  Indeed, there is a natural local duality pairing between tensors $\btheta(\XX)\in S^{2}T_{\XX}\body$ and $\bepsilon(\XX)\in S^{2}T^{\star}_{\XX}\body$, which writes
  \begin{equation*}
    \btheta(\XX) : \bepsilon(\XX).
  \end{equation*}
  Using this duality, the adjoint of the linear operator
  \begin{equation*}
    S^{2}T^{\star}_{\XX}\body \to S^{2}T^{\star}_{\XX}\body, \qquad \bepsilon(\XX) \mapsto \Gamma_{\bgamma}(\bgamma_{t},\bepsilon)(\XX) = \Gamma_{\bgamma}(\bgamma_{t},\bepsilon(\XX)),
  \end{equation*}
  denoted by $\Gamma^{\star}_{\bgamma}(\bgamma_{t},\btheta)(\XX)$, is defined implicitly by
  \begin{equation}\label{eq:GammaStar}
    \btheta(\XX) : \Gamma_{\bgamma}(\bgamma_{t},\bepsilon(\XX)) = \Gamma_{\bgamma}^{\star}(\bgamma_{t},\btheta(\XX)): \bepsilon(\XX).
  \end{equation}
  In that case, we get immediately
  \begin{equation*}
    D_{t}\btheta = \partial_{t}\btheta - \Gamma_{\bgamma}^{\star}(\bgamma_{t},\btheta).
  \end{equation*}
\end{rem}

\section{General covariance -- Material frame indifference -- Objectivity}
\label{sec:objectivity}

In Classical Mechanics, the ``physical space'' is described by a \emph{three-dimensional oriented Euclidean affine space}, whereas ``time'' is assumed to be ``absolute'' and described by a \emph{one-dimensional oriented Euclidean affine space}. These assumptions seem to be confirmed by experience in a good approximation, at least on a daily scale. Paraphrasing Jean-Marie Souriau~\cite{Sou1997}, our clocks are theoretically \emph{synchronous} and we expect that, whatever the fate of each of us, they will indicate the same time at each of our meetings and, by extension, that time will be the same everywhere. Likewise, spatial length seems to make absolute sense. The choice of a time's unit and of an orthonormal space frame thus makes it possible to locate any \emph{event} of the universe by a quadruplet of real numbers $(t,x,y,z)$ which are the coordinates of this event in this \emph{frame} and which will be written in a more condensed way in the form $(t, \xx)$. Therefore, it is traditionally assumed that a change of observer leads to a transformation
\begin{equation*}
  (\bar{t},\bar{\xx}) = (t + t_{0},g(t)\xx),
\end{equation*}
where
\begin{equation*}
  g(t)\xx = Q(t)\xx + \bc(t)
\end{equation*}
is a path of Euclidean isometries of space $\espace$, with $Q(t)\in \SO(3)$ and $\bc(t) \in \RR^{3}$.

The notion of \emph{objectivity} or \emph{material indifference} in modern language, although often confused and controversial in the mechanical literature, seems to go back to the work of Oldroyd~\cite{Old1950} and the famous treatise of Truesdell and Noll~\cite{TN1965}, which sought to formulate principles of covariance that had to be respected by constitutive laws. We will not enter into the debate on the merits of these hypotheses here~\cite{YMO2006,YM2012}, but we will seek to clarify the mathematical definition of the concept of objectivity. To do so, we will introduce the following notation. Given a path of embeddings $\widetilde{\pp} := (\pp(t))$ and a path of space's diffeomorphisms $\widetilde{\varphi} := (\varphi(t))$, we set
\begin{equation*}
  (\widetilde{\varphi} \star \widetilde{\pp})(t) := \varphi(t)\circ\pp(t),
\end{equation*}
where $\varphi(t) \in \Diff(\espace)$ and $\pp(t) \in \Emb(\espace)$. In order to define rigorously objectivity, we are lead to formulate the following definition.

\begin{defn}
  A \emph{material tensor field} is a mapping $\mathcal{F} : \widetilde{\pp} \mapsto \bt_{\widetilde{\pp}}$ which, to any path of embeddings $\widetilde{\pp} := (\pp(t))$, associates a tensor field $\bt_{\widetilde{\pp}} = (\bt_{\widetilde{\pp}}(t))$, depending on time $t$ and defined, at each time $t$, on $\Omega_{p(t)} := \pp(t)(\body)$.
\end{defn}

\begin{rem}\label{rem:embeddings-bundle}
  In more rigorous mathematical language, the mapping $\mathcal{F}$ corresponds to a smooth section along a path $\widetilde{\pp}$ (see~\autoref{sec:covariant-derivatives}) of the vector bundle
  \begin{equation*}
    \EE = \bigsqcup_{\pp \in \Emb(\body,\espace)} \EE_{\pp},
  \end{equation*}
  where
  \begin{equation*}
    \EE_{\pp} := \Cinf(\Omega_{\pp},\TT)
  \end{equation*}
  is the vector space of tensor fields of a given type $\mathbb{T}$ (a vector space of tensors on $\RR^{3}$) on $\espace$ but defined \textit{a priori} only on $\Omega_{\pp}$. This (infinite dimensional) vector bundle has two natural trivializations. The first one writes
  \begin{equation}\label{eq:E-first-trivialization}
    \Psi_{1} : \EE \to \Emb(\body,\espace) \times \Cinf(\body,\TT), \qquad \bt_{\pp} \mapsto (\pp, \bt_{\pp}\circ \pp)
  \end{equation}
  and the second one is given by
  \begin{equation}\label{eq:E-second-trivialization}
    \Psi_{2} : \EE \to \Emb(\body,\espace) \times \Gamma(\TT(\body)), \qquad \bt_{\pp} \mapsto (\pp, \pp^{*}\bt_{\pp}),
  \end{equation}
  where $\TT(\body)$ is the (finite dimensional) vector bundle of tensors of type $\TT$ over $\body$.
\end{rem}

\begin{defn}\label{def:F-is-obj-cov}
  Let $\mathcal{F} : \widetilde{\pp} \mapsto \bt_{\widetilde{\pp}}$ be a material tensor field. Then, we say that
  \begin{enumerate}
    \item $\mathcal{F}$ is \emph{objective} if
          \begin{equation*}
            \bt_{\widetilde{g} \star \widetilde{\pp}}(t) = g(t)_{*}\,\bt_{\widetilde{\pp}}(t),
          \end{equation*}
          for any path of embeddings $\widetilde{\pp}$ and any path of \emph{Euclidean isometries} $\widetilde{g} := (g(t))$ of $\espace$;
    \item $\mathcal{F}$ is \emph{general covariant} if
          \begin{equation*}
            \bt_{\widetilde{\varphi} \star \widetilde{\pp}}(t) = \varphi(t)_{*}\,\bt_{\widetilde{\pp}}(t),
          \end{equation*}
          for any path of embeddings $\widetilde{\pp}$ and any path of \emph{diffeomorphisms} $\widetilde{\varphi} := (\varphi(t))$ of $\espace$.
  \end{enumerate}
  Here, $g(t)_{*}$ and $\varphi(t)_{*}$ mean the \emph{push-forward} by $g(t)$ or $\varphi(t)$ (see~\autoref{sec:covariant-derivatives}).
\end{defn}

It is obvious that any general covariant mapping $\mathcal{F}$ is also objective, but the converse is not true.

\begin{rem}
  Objectivity (or general covariance) translates into equivariant properties of sections along a path of the vector bundle $\EE$ (see~\autoref{sec:pullback-pushforward}).
\end{rem}

Let us illustrate this concept with two classical examples~\cite{MH1994,Hau2002,Ber2012}. Let $\mathcal{F} : \widetilde{\pp} \mapsto \uu_{\widetilde{\pp}}$ be the mapping which, to any path of embeddings $\widetilde{\pp}$, associates its (right) Eulerian velocity
\begin{equation*}
  \uu_{\widetilde{\pp}}(t) := (\partial_{t}\pp)\circ \pp(t)^{-1},
\end{equation*}
which is a vector field on $\Omega_{p(t)}$. We have then
\begin{equation*}
  \uu_{\widetilde{g} \star \widetilde{\pp}}(t) = g(t)_{*}\uu_{\widetilde{\pp}}(t) + \ww(t),
\end{equation*}
where $\ww(t) := (\partial_{t}g)\circ g(t)^{-1}$ is the \emph{drive velocity} and $g(t)_{*}\uu_{\widetilde{\pp}}(t)$ is the push-forward of the Eulerian velocity by $g(t)$. Thus, the Eulerian velocity is not objective.

Similarly, the Eulerian velocity gradient transforms as
\begin{equation*}
  \nabla \uu_{\widetilde{g} \star \widetilde{\pp}}(t) = Q(t) (\nabla \uu_{\widetilde{\pp}}(t)) Q(t)^{t} + \bOmega(t),
\end{equation*}
where $g(t)\xx = Q(t)\xx + \bc(t)$ and $\bOmega(t) := Q_{t}(t)Q(t)^{-1}$ is skew symmetric. The mapping $\widetilde{\pp} \mapsto \nabla \uu_{\widetilde{\pp}}$
is therefore not objective, but the rate of deformation
\begin{equation*}
  \widehat{\bd}_{\widetilde{\pp}} := \frac{1}{2}\left( \nabla \uu_{\widetilde{\pp}} + (\nabla \uu_{\widetilde{\pp}})^{t}\right),
\end{equation*}
is objective, because
\begin{equation*}
  \widehat{\bd}_{\widetilde{g} \star \widetilde{\pp}}(t) = \frac{1}{2} \left(\nabla \uu_{\widetilde{g} \star \widetilde{\pp}}(t) + (\nabla \uu_{\widetilde{g} \star \widetilde{\pp}}(t))^{t}\right) =  \frac{1}{2} Q(t)\left(\nabla \uu_{\widetilde{\pp}}(t) + (\nabla \uu_{\widetilde{\pp}}(t))^{t}\right)Q(t)^{t} = g(t)_{*}\widehat{\bd}_{\widetilde{\pp}}(t),
\end{equation*}
since $\bOmega(t)^{t} = - \bOmega(t)$. It is nevertheless not general covariant.

\begin{exam}
  Let $\bt$ be a tensor field defined on $\espace$ and $\mathcal{F}(\widetilde{\pp}) = \bt_{\widetilde{\pp}}$, be the restriction of $\bt$ to the deformed configuration $\Omega_{p(t)} = \pp(t)(\body)$ at time $t$. Then, $\mathcal{F}$ is objective, if and only if, $g_{*}\bt = \bt$ for each isometry $g$. When $\bt$ is a scalar function, this implies that it is constant. If $\bt$ is a vector field, then $\bt =0$. If $\bt$ is a field of covariant symmetric second-order tensors, then $\bt = \lambda \bq$, where $\lambda$ is a constant and $\bq$ is the Euclidean metric. If $\bt$ is a field of alternate covariant tensors of order $3$, then $\bt = \lambda \vol_{\bq}$ is proportional to the canonical Euclidean volume form by a constant.
\end{exam}

\begin{exam}
  Let $\bT$ be a tensor field defined on $\body$ and $\mathcal{F}(\widetilde{\pp}) = (\pp(t)_{*}\bT)$. Then $\mathcal{F}$ is general covariant and therefore objective. This is the case, for instance, of the push-forward of the mass measure $\mu$ on space. We deduce from this fact, and from the objectivity of $\vol_{\bq}$, the objectivity of the mass density $\rho$. The mass density is however not general covariant but more than just objective. It is covariant under paths of volume-preserving diffeomorphisms of $\espace$, that is diffeomorphisms such that $\varphi_{*}\vol_{\bq}=\vol_{\bq}$.
\end{exam}

The notion of objectivity (and of general covariance) extends, without difficulty, from tensor fields to \emph{tensor-distributions}, \emph{i.e.} to continuous linear functionals on tensor fields. More precisely, let $\mathcal{P}_{\widetilde{\pp}}$ be a path of tensor-distributions over the space of symmetric second-order covariant tensor fields and $\widetilde{g}$ be a path of Euclidean isometries. We will say that $\mathcal{P}_{\widetilde{\pp}}$ is objective if
\begin{equation*}
  \mathcal{P}_{\widetilde{g} \star \widetilde{\pp}}(t) = g(t)_{*} \mathcal{P}_{\widetilde{\pp}}(t),
\end{equation*}
where
\begin{equation*}
  (g(t)_{*} \mathcal{P}_{\widetilde{\pp}}(t))(\bk) := \mathcal{P}_{\widetilde{\pp}}(t)(g(t)^{*}\bk),
\end{equation*}
for any symmetric second-order covariant tensor fields $\bk$ defined on $\Omega_{\pp(t)}$. This extended definition allows us to reformulate the following result known as \emph{Noll's theorem}~\cite{Nol1974,TN1965}.

\begin{thm}
  The tensor-distribution with density
  \begin{equation*}
    \mathcal{P}_{\widetilde{\pp}}(\bk) := \int_{\Omega_{\pp}} \bsigma_{\widetilde{\pp}} : \bk \, \vol_{\bq}
  \end{equation*}
  is objective if and only if the tensor field $\bsigma_{\widetilde{\pp}}$ is.
\end{thm}

\begin{proof}
  The tensor-distribution $\mathcal{P}_{\widetilde{\pp}}$ is objective if and only if
  \begin{equation*}
    \mathcal{P}_{\widetilde{g} \star \widetilde{\pp}}(t) = g(t)_{*} \mathcal{P}_{\widetilde{\pp}}(t),
  \end{equation*}
  which writes
  \begin{equation*}
    \int_{\Omega_{\bar{\pp}(t)}} \bsigma_{\widetilde{g} \star\widetilde{\pp}}(t) : \bar{\bk} \, \vol_{\bq} = \int_{\Omega_{p(t)}} \bsigma_{\widetilde{\pp}}(t) : (g(t)^{*}\bar{\bk}) \, \vol_{\bq} ,
  \end{equation*}
  for any field $\bar{\bk}$ defined on $\Omega_{\bar{\pp}(t)}$, where $\bar{\pp}(t) = g(t) \circ \pp(t)$. But
  \begin{equation*}
    \bsigma_{\widetilde{\pp}}(t) : (g(t)^{*}\bar{\bk}) \, \vol_{\bq} = g(t)^{*}\left(g(t)_{*}\bsigma_{\widetilde{\pp}}(t) : \bar{\bk} \, \vol_{\bq}\right)
  \end{equation*}
  since $g(t)^{*}\vol_{\bq} = \vol_{\bq}$. The objectivity of $\mathcal{P}_{\widetilde{\pp}}$ therefore translates, after using the change of variables formula, into
  \begin{equation*}
    \int_{\Omega_{\bar{\pp}(t)}} \bsigma_{\widetilde{g} \star \widetilde{\pp}}(t) : \bar{\bk} \, \vol_{\bq} = \int_{\Omega_{\bar{\pp}(t)}} g(t)_{*}\bsigma_{\widetilde{\pp}}(t) : \bar{\bk} \, \vol_{\bq}.
  \end{equation*}
  This being true for any field $\bar{\bk}$ defined on $\Omega_{\bar{\pp}}$, we have the equivalence between
  \begin{equation*}
    \mathcal{P}_{\widetilde{g} \star \widetilde{\pp}}(t) = g(t)_{*} \mathcal{P}_{\widetilde{\pp}}(t)
  \end{equation*}
  and
  \begin{equation*}
    \bsigma_{\widetilde{g} \star \widetilde{\pp}}(t) = g(t)_{*} \bsigma_{\widetilde{\pp}}(t),
  \end{equation*}
  which completes the proof.
\end{proof}

The next result states that each Cauchy elastic constitutive law as defined by Rougée is necessarily \emph{objective}.

\begin{thm}\label{thm:obj-elas}
  Consider a vector field $F:\Met(\body)\to T\Met(\body)$ and the corresponding elastic constitutive law in the sense of Rougée: $\btheta = \bgamma^{-1}F(\bgamma)\bgamma^{-1}$. Then the resulting Cauchy elastic law $\bsigma = \bsigma_{\pp}$, $\rho=\rho_{\pp}=(\pp_*\mu)/ \vol_\bq$, where
  \begin{equation*}
    \bsigma_{\pp} = \rho_{\pp}\, \bq^{-1}\pp_{*}F(\pp^{*}\bq)\bq^{-1},
  \end{equation*}
  is objective.
\end{thm}

\begin{proof}
  Given a path of embeddings $\widetilde{\pp} := (\pp(t))$, we have
  \begin{equation*}
    \bsigma_{\widetilde{\pp}} = \rho_{\widetilde{\pp}} \, \bq^{-1}\bepsilon_{\widetilde{\pp}}\bq^{-1},
  \end{equation*}
  where $\bepsilon_{\widetilde{\pp}}(t) = \pp(t)_{*}F(\pp(t)^{*}\bq)$. But
  \begin{equation*}
    (g(t) \circ \pp(t))^{*}\bq = \pp(t)^{*}(g(t)^{*}\bq) = \pp(t)^{*}\bq,
  \end{equation*}
  since $g(t)$ is an isometry of $\bq$, and hence $\bepsilon_{\widetilde{g}\star\widetilde{\pp}}(t) =  g(t)_{*}\bepsilon_{\widetilde{\pp}}(t)$. Moreover
  \begin{equation*}
    \rho_{\widetilde{g}\star\widetilde{\pp}}(t)\vol_{\bq} = (g(t) \circ \pp(t))_{*}\mu = g(t)_{*}(\pp(t)_{*}\mu) = (g(t)_{*}\rho_{\widetilde{\pp}}(t))\vol_{\bq},
  \end{equation*}
  since $g(t)_{*}\vol_{\bq} = \vol_{\bq}$, and thus $\rho_{\widetilde{g}\star\widetilde{\pp}}(t) = g(t)_{*}\rho_{\widetilde{\pp}}(t)$. We get finally
  \begin{equation*}
    \bsigma_{\widetilde{\pp}}(t) = (g(t)_{*}\rho_{\widetilde{\pp}}(t))\bq^{-1}(g(t)_{*}\bepsilon_{\widetilde{\pp}}(t))\bq^{-1} = g(t)_{*}\bsigma_{\widetilde{\pp}}(t).
  \end{equation*}
\end{proof}

\section{Material and objective derivatives}
\label{sec:material-derivatives}

Objectivity has been extended to time derivatives. These \emph{objective time derivatives}, noted $d_{\widetilde{\pp}}/dt$ here, are principally used to formulate hypo-elasticity laws on the deformed configuration~\cite{Jau1911,Tru1955}. For instance, in solid mechanics, one often writes
\begin{equation}\label{eq:hypo}
  \dmat{\widetilde{\pp}}{\btau} = \mathbf{H} : \bd^{e} = \mathbf{H} : ( \bd - \bd^{p}),
  \qquad
  \left(\dmat{\widetilde{\pp}}{\tau}^{ij} = H^{ijkl} d_{kl} = H^{ijkl} (d_{ij} - d_{ij}^{p}) \right),
\end{equation}
where $\btau$ is the Kirchhoff stress tensor, $\bd$ is the total strain rate, and where $\bd^{e} = \bd - \bd^{p}$ and $\bd^{p}$ are respectively the elastic and the plastic strain rates. The fourth-order tensor $\mathbf{H}$, called the \emph{hypo-elasticity tensor}, depends \emph{a priori} on the state of deformation~\cite{Ber1960,SP1984}. In fluid mechanics, a relative elasticity of the fluid medium is introduced by writing for instance~\cite{Old1950,Hau2002}:
\begin{equation}\label{eq:hypofluides}
  \bsigma + \lambda \dmat{\widetilde{\pp}}{\bsigma} = 2\eta\, \bq^{-1} \bd\,\bq^{-1},
\end{equation}
where $\eta$ is the dynamic viscosity, and $\lambda$ is the relaxation time.

Although the concept of \emph{objective time derivative} is the subject of an abundant literature in Continuum Mechanics, it rarely seems to be defined with enough mathematical rigour. In this section, we aim to fill this gap. First, we will formulate the concept of \emph{material time derivative} on material tensor fields (which has not to be confused with the \emph{particle derivative} as defined in~example~\ref{exam:particle-derivative}). In mathematical terms, a material time derivative is nothing else than a \emph{covariant derivative} (along a path) on the vector bundle (of infinite dimension) $\EE$ defined in remark~\ref{rem:embeddings-bundle}. This point of view will be detailed in \autoref{sec:converse-theorem}. However, in order to be more readable by the mechanical community we will introduce the following definition.

\begin{defn}
  A \emph{material derivative} is a \emph{linear operator} $d_{\widetilde{\pp}}/dt$ acting on the space of material tensor fields $\bt_{\widetilde{p}}$, defined along each path of embeddings $\widetilde{\pp}: t \mapsto \pp(t)$, and furthermore satisfying the Leibniz rule
  \begin{equation*}
    \dmat{\widetilde{\pp}}{}(f\bt_{\widetilde{p}}) = (\partial_{t}f)\bt_{\widetilde{p}} + f \dmat{\widetilde{\pp}}{}(\bt_{\widetilde{p}}),
  \end{equation*}
  for each function $f : t \mapsto f(t)$.
\end{defn}

\begin{exam}[Particle derivative]\label{exam:particle-derivative}
  Perhaps, the best known example of a material derivative is the \emph{particle derivative}, defined as follows. Let $\widetilde{\pp}: t \mapsto \pp(t)$ be a path of embeddings and $\bt_{\widetilde{\pp}}$ be a tensor field defined along $\widetilde{\pp}$. For each time $t$, $\bt_{\widetilde{\pp}}(t)$ is a tensor field on $\Omega_{p(t)}$, and we will write
  \begin{equation*}
    \bt(t,\xx) := \bt_{\widetilde{\pp}}(t)(\xx), \qquad \xx \in \Omega_{\pp(t)}.
  \end{equation*}
  Now, for each particle indexed by $\XX \in \body$, its ``history'' is described by the curve $\tilde{\xx}: t \mapsto \pp(t,\XX)$ on $\espace$ and $\dot{\tilde{\xx}}(t) = \uu(t,\tilde{\xx}(t))$, where $\uu$ is the Eulerian velocity. Then,
  \begin{equation*}
    \bt(t) := \bt(t,\tilde{\xx}(t))
  \end{equation*}
  is a tensor field in $\espace$, defined along the path $\tilde{\xx}$. The \emph{particle derivative} of $\bt$ is defined as the pointwise derivative
  \begin{equation}\label{eq:particular-derivative}
    \dot{\bt} := \partial_{t}(\bt \circ \pp)\circ \pp^{-1} = \partial_{t}\bt + \nabla_{\uu}\bt,
  \end{equation}
  where $\nabla$ is the canonical derivative on the Euclidean space $\espace$.
\end{exam}

\begin{rem}
  The particle derivative corresponds to the canonical covariant derivative (see remark~\ref{rem:canonical-covariant-derivative}) on the vector bundle $\EE$ associated with the first trivialization~\eqref{eq:E-first-trivialization} $\Emb(\body,\espace) \times \Cinf(\body,\TT)$ described in remark~\ref{rem:embeddings-bundle}.
\end{rem}

The extension of the concept of objectivity for material derivatives is naturally obtained by requiring that they transform objective quantities into objective quantities. This leads us to formulate the following definitions. We recall that $\widetilde{g} \star \widetilde{\pp}= (\varphi(t) \circ \pp(t))$.

\begin{defn}\label{def:mat-deriv-obj}
  (1) A \emph{material derivative} $d_{\widetilde{\pp}}/dt$ is \emph{objective} if
  \begin{equation}\label{eq:objectivity}
    \left(\dmat{\widetilde{g} \star \widetilde{\pp}}{}\bt_{\widetilde{g} \star \widetilde{\pp}}\right)(t) = g(t)_{*}\left(\dmat{\widetilde{\pp}}{}\bt_{\widetilde{\pp}}\right)(t),
  \end{equation}
  for each path of embeddings $\widetilde{\pp}=(\pp(t))$, each path of Euclidean isometries $\widetilde{g}=(g(t))$, and each material tensor field $\bt_{\widetilde{\pp}}$ defined along $\widetilde{\pp}$.

  (2) A \emph{material derivative} $d_{\widetilde{\pp}}/dt$ is \emph{general covariant} if
  \begin{equation}\label{eq:general-covariance}
    \left(\dmat{\widetilde{\varphi} \star \widetilde{\pp}}{}\bt_{\widetilde{\varphi} \star \widetilde{\pp}}\right)(t) = \varphi(t)_{*}\left(\dmat{\widetilde{\pp}}{}\bt_{\widetilde{\pp}}\right)(t),
  \end{equation}
  for each path of embeddings $\widetilde{\pp}=(\pp(t))$, each path of diffeomorphisms $\widetilde{\varphi} = (\varphi(t))$ of space, and each material tensor field $\bt_{\widetilde{\pp}}$ defined along $\widetilde{\pp}$.
\end{defn}

\begin{exam}[Non-objectivity of the particle derivative]
  The particle derivative defined by~\eqref{eq:particular-derivative} is not objective. Indeed, if $\widetilde{g} \star \widetilde{\pp} = (g(t) \circ \pp(t))$, its Eulerian velocity writes $\bar \uu= g_{*} \uu + \ww$, where $\uu$ is the Eulerian velocity of the path $\widetilde{\pp} = (\pp(t))$ and $\ww := \partial_{t}g \circ g^{-1}$ is the drive velocity. We get thus, thanks to remark~\ref{rem:Lie-path},
  \begin{equation*}
    \dmat{\widetilde{g} \star \widetilde{\pp}}{}(\widetilde{g}_{*}\bt) = \widetilde{g}_{*}\left( \partial_{t}\bt + \nabla_{\uu}\bt + \nabla_{\widetilde{g}^{*}\ww}\bt- \Lie_{\widetilde{g}^{*}\ww}\bt \right),
  \end{equation*}
  which is not equal to
  \begin{equation*}
    \widetilde{g}_{*}\left(\dmat{\widetilde{\pp}}{\bt}\right) = \widetilde{g}_{*}\left( \partial_{t}\bt + \nabla_{\uu}\bt\right).
  \end{equation*}
\end{exam}

\begin{exam}[General covariance of the Lie derivative]\label{exam:Lie-derivative-general-covariance}
  Another example of material time derivative is given by
  \begin{equation}\label{eq:Oldroyd-derivative}
    \dmat{\widetilde{\pp}}{\bt} = \overset{\triangledown}{\bt} := \partial_{t}\bt + \Lie_{\uu}\bt.
  \end{equation}
  This one is not only objective but also general covariant~\cite{MH1994}. Indeed
  \begin{equation*}
    \dmat{\widetilde{\varphi} \star \widetilde{\pp}}{}(\widetilde{\varphi}_{*}\bt) = \partial_{t}(\widetilde{\varphi}_{*}\bt) + \Lie_{\bar{\uu}}(\widetilde{\varphi}_{*}\bt).
  \end{equation*}
  where $\bar{\uu}$ is the Eulerian velocity of the path of embeddings $\widetilde{\varphi} \star \widetilde{\pp} = (\varphi(t) \circ \pp(t))$. But $\bar{\uu} = \widetilde{\varphi}_{*}\uu + \ww$, where $\ww := \partial_{t} \varphi \circ \varphi^{-1}$ is the drive velocity. We get thus (by remark~\ref{rem:Lie-path} and~\eqref{eq:pullback-and-Lie-derivative})
  \begin{equation*}
    \dmat{\widetilde{\varphi} \star \widetilde{\pp}}{}(\widetilde{\varphi}_{*}\bt) = \widetilde{\varphi}_{*}(\partial_{t}\bt) - \Lie_{\ww}(\widetilde{\varphi}_{*}\bt) + \widetilde{\varphi}_{*}(\Lie_{\uu}\bt) + \Lie_{\ww}(\widetilde{\varphi}_{*}\bt) = \widetilde{\varphi}_{*}\left( \partial_{t}\bt + \Lie_{\uu}\bt\right) = \widetilde{\varphi}_{*}\left(\dmat{\widetilde{\pp}}{\bt}\right).
  \end{equation*}
\end{exam}

The material derivative $\overset{\triangledown}{\bt}$ defined by~\eqref{eq:Oldroyd-derivative} corresponds to the canonical covariant derivative (see remark~\ref{rem:canonical-covariant-derivative}) on the vector bundle $\EE$ associated with the second trivialization~\eqref{eq:E-second-trivialization} $\Emb(\body,\espace) \times \Gamma(\TT(\body))$ described in remark~\ref{rem:embeddings-bundle}. It can thus be recast using lemma~\ref{lem:fundamental-time-derivative-formula} as
\begin{equation}\label{eq:canonical-material-derivative}
  \dmat{\widetilde{\pp}}{\bt} := \widetilde{\pp}_{*}\left(\partial_{t}(\widetilde{\pp}^{*}\bt)\right),
\end{equation}
which allows us for the following fundamental observation: \emph{When symmetric second-order covariant tensor fields are involved, formula~\eqref{eq:canonical-material-derivative} can be interpreted as the push-forward on the deformed configuration of the canonical covariant derivative~\eqref{eq:canonical-covariant-derivative} on $T\Met(\body)$, the manifold of Riemannian metrics}. As there is no reason to limit this interpretation to the canonical covariant derivative on $T\Met(\body)$, this gives us the possibility to produce this way, an infinity of material time derivatives for second-order covariant tensor fields. As detailed in~\autoref{subsec:covariant-derivative-on-TMet}, any covariant derivative on
\begin{equation*}
  T\Met(\body) = \Met(\body) \times \Gamma(S^{2}T^{\star}\body)
\end{equation*}
writes
\begin{equation*}
  D_{t}\bepsilon = \partial_{t} \bepsilon + \Gamma_{\bgamma}(\bgamma_{t},\bepsilon),
\end{equation*}
where $\Gamma_{\bgamma}$ is a bilinear operator that depends on $\bgamma$. We then define a material time derivative on \emph{second-order covariant tensor fields}, by setting
\begin{equation}\label{eq:objective-derivatives-cov}
  \dmat{\widetilde{\pp}}{\bk} := \widetilde{\pp}_{*}\left(D_{t}(\widetilde{\pp}^{*}\bk)\right) = \widetilde{\pp}_{*}\left(\partial_{t}(\widetilde{\pp}^{*}\bk)\right) + \widetilde{\pp}_{*}\left(\Gamma_{\widetilde{\pp}^{*}\bq}(\partial_{t}(\widetilde{\pp}^{*}\bq),\widetilde{\pp}^{*}\bk)\right),
\end{equation}
where $\widetilde{\pp} = (\pp(t))$ is a path of embeddings, $\uu$ its Eulerian velocity, and $\bk$ is a tensor field defined along $\widetilde{\pp}$.

\begin{thm}\label{thm:objective-derivatives}
  Let $D$ be a covariant derivative on $T\Met(\body)$. Then, the material time derivative on symmetric second-order covariant tensor fields $\bk$ induced by $D$ and given by~\eqref{eq:objective-derivatives-cov} is \emph{objective}.
\end{thm}

\begin{rem}
  It should be noted, however, that these objective material derivatives have no reason to be, in general, covariant with respect to a non-rigid motion of space.
\end{rem}

\begin{proof}
  We already know that
  \begin{equation*}
    \overset{\triangledown}{\bk} = \widetilde{\pp}_{*}\left(\partial_{t}(\widetilde{\pp}^{*}\bk)\right) = \partial_{t}\bt + \Lie_{\uu}\bk
  \end{equation*}
  is an objective time derivative (and is even general covariant, see example~\ref{exam:Lie-derivative-general-covariance}), so we only need to show that
  \begin{equation*}
    (\widetilde{g} \star \widetilde{\pp})_{*} \left(\Gamma_{(\widetilde{g} \star \widetilde{\pp})^{*}\bq}(\partial_{t}((\widetilde{g} \star \widetilde{\pp})^{*}\bq),(\widetilde{g} \star \widetilde{\pp})^{*}(\widetilde{g}_{*}\bk))\right) = \widetilde{g}_{*} \left(\widetilde{\pp}_{*}\left(\Gamma_{\widetilde{\pp}^{*}\bq}(\partial_{t}(\widetilde{\pp}^{*}\bq),\widetilde{\pp}^{*}\bk)\right)\right),
  \end{equation*}
  but this is true since
  \begin{equation*}
    (\widetilde{g} \star \widetilde{\pp})_{*} = \widetilde{g}_{*} \widetilde{\pp}_{*}, \qquad (\widetilde{g} \star \widetilde{\pp})^{*} = \widetilde{\pp}^{*} \widetilde{g}^{*},
  \end{equation*}
  and $(\widetilde{g} \star \widetilde{\pp})^{*}\bq = \widetilde{\pp}^{*}\bq$, regardless of the path $\widetilde{g}$ of Euclidean isometries in $\espace$.
\end{proof}

Theorem~\ref{thm:objective-derivatives} extends to \emph{second-order contravariant tensor fields}, when the covariant derivative $D$ induced on $T^{\star}\Met(\body)$ (by Leibniz rule) preserves distributions with density. Indeed, in that case, an operator $D_{t}\btheta$ has been introduced for symmetric second-order contravariant tensor fields $\btheta$ on $\body$ (see~\autoref{subsec:covariant-derivative-on-TMet}), which is defined implicitly by the relation
\begin{equation*}
  D_{t}\btheta : \bepsilon+\btheta : D_{t}\bepsilon=\partial_{t}\left(\btheta : \bepsilon\right).
\end{equation*}
This allows to define a material time derivative on \emph{second-order contravariant tensor fields}, by setting
\begin{equation}\label{eq:objective-derivatives-contra}
  \dmat{\widetilde{\pp}}{\btau} := \widetilde{\pp}_{*}\left(D_{t}(\widetilde{\pp}^{*}\btau)\right),
\end{equation}
where $\widetilde{\pp} = (\pp(t))$ is a path of embeddings, $\uu$ is its Eulerian velocity, and $\btau$ is a tensor field defined along $\widetilde{\pp}$. In that case, we have the following \emph{pseudo-Leibniz rule} between objective derivatives of covariant and contravariant symmetric second-order tensor fields, induced by the same covariant derivative $D$ on $\Met(\body)$,
\begin{equation}\label{eq:pseudo-Leibniz-rule}
  \dmat{\widetilde{\pp}}{\btau} : \bk + \btau : \dmat{\widetilde{\pp}}{\bk} = \partial_{t}(\btau :\bk) + \Lie_{\uu}(\btau :\bk) = \widetilde{\pp}_{*}\left(\partial_{t}(\widetilde{\pp}^{*}(\btau :\bk))\right).
\end{equation}

\begin{thm}\label{thm:objectivity-covectors}
  A covariant derivative $D$ on $T\Met(\body)$, which preserves distributions with density, induces an \emph{objective material derivative} on second-order contravariant tensor fields $\btau$, which writes
  \begin{equation*}
    \dmat{\widetilde{\pp}}{\btau} := \widetilde{\pp}_{*}\left(D_{t}(\widetilde{\pp}^{*}\btau)\right),
  \end{equation*}
  where $D_{t}\btheta$ is defined implicitly by the rule
  \begin{equation}\label{eq:Dtheta-Gamma}
    D_{t}\btheta : \bepsilon = \partial_{t}\btheta : \bepsilon - \btheta : \Gamma_{\bgamma}(\bgamma_{t},\bepsilon),
  \end{equation}
  for all second-order covariant tensor fields $\bepsilon$ on $\body$ and all second-order contravariant tensor fields $\btheta$ on $\body$.
\end{thm}

\begin{proof}[Proof of theorem~\ref{thm:objectivity-covectors}]
  Let $\widetilde{\pp}= (\pp(t))$ be a path of embeddings and $\widetilde{g}=(g(t))$, a path of Euclidean isometries. We have to show that
  \begin{equation*}
    \dmat{\widetilde{g} \star \widetilde{\pp}}{(\widetilde{g}_{*}\btau)} = \widetilde{g}_{*}\left(\dmat{\widetilde{\pp}}{\btau}\right),
  \end{equation*}
  for all symmetric second-order contravariant vector fields $\btau$, defined along $\widetilde{\pp}$. By virtue of~\eqref{eq:pseudo-Leibniz-rule} and the general covariance of the material derivative
  \begin{equation*}
    \widetilde{\pp}_{*}\left(\partial_{t}(\widetilde{\pp}^{*}(\btau :\bk))\right) = \partial_{t}(\btau :\bk) + \Lie_{\uu}(\btau :\bk),
  \end{equation*}
  we have
  \begin{align*}
    \dmat{\widetilde{g} \star \widetilde{\pp}}{}(\widetilde{g}_{*}\btau) : \widetilde{g}_{*}\bk & = \widetilde{g}_{*}(\partial_{t}(\btau :\bk) + \Lie_{\uu}(\btau :\bk)) -  \widetilde{g}_{*}\btau : \dmat{\widetilde{g} \star \widetilde{\pp}}{}(\widetilde{g}_{*}\bk)
    \\
                                                                                                & = \widetilde{g}_{*}(\partial_{t}(\btau :\bk) + \Lie_{\uu}(\btau :\bk)) - \widetilde{g}_{*}\btau : \widetilde{g}_{*} \left(\dmat{\widetilde{\pp}}{\bk}\right)
    \\
                                                                                                & = \widetilde{g}_{*} \left( \partial_{t}(\btau :\bk) + \Lie_{\uu}(\btau :\bk) - \btau :  \left(\dmat{\widetilde{\pp}}{\bk}\right) \right)
    \\
                                                                                                & = \widetilde{g}_{*}\left(\dmat{\widetilde{\pp}}{\btau} :\bk \right) = \widetilde{g}_{*}\left(\dmat{\widetilde{\pp}}{\btau}\right) : \widetilde{g}_{*}\bk,
  \end{align*}
  for any $\btau$ and $\bk$, and thus
  \begin{equation*}
    \dmat{\widetilde{g} \star \widetilde{\pp}}{}(g_{*}\btau)  = \widetilde{g}_{*} \left(\dmat{\widetilde{\pp}}{\btau}\right).
  \end{equation*}
\end{proof}

\emph{Distributions with density} in $T^{\star}_{\bgamma}\Met(\body)$ could have been formulated using the Riemannian volume $\vol_{\bgamma}$, rather than the mass measure $\mu$. In other words, writing
\begin{equation*}
  \mathcal{P}(\bepsilon) = \int_{\body} (\msigma:\bepsilon) \vol_{\bgamma},
\end{equation*}
instead of
\begin{equation*}
  \mathcal{P}(\bepsilon) = \int_{\body} (\btheta:\bepsilon) \mu,
\end{equation*}
with $\mu =\rho_{\bgamma} \,\vol_{\bgamma}$, $\rho_{\bgamma} = \pp^{*}\rho$, $\msigma:=p^{*}\bsigma$ (the Noll stress tensor) and $\btheta:=p^{*}\btau$ (the Rougée stress tensor). Note that each covariant derivative $D$ on $T\Met(\body)$, which preserves distributions with density relative to the mass measure $\mu$, preserves also distributions with density relative to the volume measure $\vol_{\bgamma}$, and \textit{vice versa}. Moreover, we have expression~\eqref{eq:Dtheta-Gamma} for $D_{t}\btheta$, while we have

\begin{equation}\label{eq:volume-derivative-density}
  \overline{D}_{t}\msigma := D_{t}\msigma + \frac{1}{2} \tr(\bgamma^{-1}\bgamma_{t})\, \msigma,
\end{equation}
for all second-order contravariant tensor fields $\msigma$ defined on $\body$. If one prefers to keep the definition of a distribution with density relatively to the mass measure $\mu$, then expression~\eqref{eq:volume-derivative-density} corresponds to the following covariant derivative on $T\Met(\body)$
\begin{equation*}
  \overline{D}_{t}\bepsilon := D_{t}\bepsilon - \frac{1}{2} \tr(\bgamma^{-1}\bgamma_{t})\bepsilon,
\end{equation*}
and
\begin{equation*}
  \overline{D}_{t}\msigma : \bepsilon+\msigma : \bar{D}_{t}\bepsilon=\partial_{t}\left(\msigma : \bepsilon\right).
\end{equation*}
From these observations, we deduce that for each covariant derivative $D$ on $T\Met(\body)$, which preserves distributions with density, we obtain two \emph{objective material derivatives} on symmetric contravariant second-order tensor fields. The first one writes
\begin{equation*}
  \dmat{\widetilde{\pp}}{\btau} := \pp_{*}\left(D_{t}(\pp^{*}\btau)\right),
\end{equation*}
and the second one writes
\begin{equation*}
  \frac{\bar{d}_{\widetilde{\pp}}\bsigma}{dt} := \pp_{*}\left(\bar{D}_{t}(\pp^{*}\bsigma)\right) = \dmat{\widetilde{\pp}}{\bsigma} + (\dive \uu)\bsigma,
\end{equation*}
where $\uu$ is the (right) Eulerian velocity, because
\begin{equation*}
  \pp_{*}\left(\frac{1}{2} \tr(\bgamma^{-1}\bgamma_{t})\msigma\right) = \tr\left(\bq^{-1}\bd\right)\bsigma = \left(\dive \uu\right)\bsigma.
\end{equation*}

\begin{rem}\label{rem:Truesdell-trick}
  As observed by Truesdell~\cite{Tru1955} for the special case $D_{t} = \partial_{t}$, this second objective material derivative can be recast as
  \begin{equation*}
    \frac{\bar{d}_{\widetilde{\pp}}\bsigma}{dt} = \rho \dmat{\widetilde{\pp}}{}\left(\frac{\bsigma}{\rho}\right),
  \end{equation*}
  since
  \begin{equation*}
    \rho\, \pp_{*}\left(\partial_{t}\pp^{*} \left(\frac{1}{\rho}\right)\right) = \dive \uu,
  \end{equation*}
  due to mass conservation $\rho_{t} + \dive(\rho \uu) = 0$.
\end{rem}

\section{A converse theorem for local objective derivatives}
\label{sec:converse-theorem}

So far, we have proved, that each covariant derivative on $T\Met(\body)$ induces an objective derivative on symmetric second-order tensor fields. More precisely, given any covariant derivative $D$ on $T\Met(\body)$, it induces a material time derivative
\begin{equation*}
  \dmat{\widetilde{\pp}}{\bk_{\widetilde{\pp}}} = \widetilde{\pp}_{*}\left(D_{t}( \widetilde{\pp}^{*}\bk_{\widetilde{\pp}})\right),
\end{equation*}
on symmetric covariant second-order tensor fields $\bk_{\widetilde{\pp}}$ defined along a path of embeddings $\widetilde{\pp}$. This material time derivative is moreover objective (theorem~\ref{thm:objective-derivatives}), which means that
\begin{equation*}
  \left(\dmat{\widetilde{g} \star \widetilde{\pp}}{}\bk_{\widetilde{g} \star \widetilde{\pp}}\right)(t) = g(t)_{*}\left(\dmat{\widetilde{\pp}}{}\bk_{\widetilde{\pp}}\right)(t),
\end{equation*}
for each path of Euclidean isometries $\widetilde{g}=(g(t))$. This covariant derivative on $T\Met(\body)$ induces moreover a covariant derivative on $T^{\star}\Met(\body)$, using the Leibniz rule, and if it preserves distributions with densities (which is in practice always the case), it induces an objective derivative on symmetric contravariant second-order tensor fields
\begin{equation*}
  \dmat{\widetilde{\pp}}{\btau_{\widetilde{\pp}}} = \widetilde{\pp}_{*}\left(D_{t}(\widetilde{\pp}^{*}{\btau_{\widetilde{\pp}}})\right).
\end{equation*}

But what about the converse? In other words, given an objective material time derivative on symmetric second-order covariant tensor fields, is it always induced by a covariant derivative on $T\Met(\body)$? At this level of generality, it is highly improbable that the converse is true. One argument in favor of this claim is that the Nash mapping
\begin{equation*}
  \mathcal{N} : \Emb(\body,\espace) \to \Met(\body), \qquad p \mapsto \bgamma = \pp^{*}\bq
\end{equation*}
is far from being surjective, all metrics $\bgamma$ in its range having vanishing curvature. The question is thus, are we able to describe all known objective rates by a covariant derivative on $T\Met(\body)$? To answer this question, we have first to formulate the problem accurately. This will lead us to prove a partial -- but fully mechanistic -- converse of theorem~\ref{thm:objective-derivatives}.

As we have already stated, each material time derivative -- therefore each objective derivative -- is in fact a covariant derivative $D$ on the vector bundle $\EE$ described in remark~\ref{rem:embeddings-bundle}. Consider now the special case, where the tensor type $\TT = {S^{2}\RR^{3}}^{\star}$ is chosen to be the vector space of symmetric covariant second-order tensors on $\RR^{3}$. This vector bundle $\EE$, with base space $\Emb(\body,\espace)$, is thus the union of vector spaces, noted $\EE_{\pp} = \Cinf(\Omega_{\pp}, {S^{2}\RR^{3}}^{\star})$, of symmetric covariant second-order tensor fields defined on $\Omega_{\pp} = \pp(\body)$. On this vector bundle, there is a preferred covariant derivative (see remark~\ref{rem:canonical-covariant-derivative} and remark~\ref{rem:embeddings-bundle}), the canonical covariant derivative associated with the second trivialization
\begin{equation}\label{eq:second-triv}
  \Psi_{2}: \: \EE \to \Emb(\body,\espace) \times \Gamma({S^{2}\RR^3}^\star),
  \qquad
  \Psi_{2}(\bt_{\pp}) = (\pp, \pp^{*}\bt_{\pp}),
\end{equation}
which writes
\begin{equation*}
  \overset{\triangledown}{\bk} := \pp_{*} \partial_{t} (\pp^{*} \bk) = \partial_{t}\bk + \Lie_{\uu}\bk,
\end{equation*}
and corresponds to the Lie rate (or Oldroyd rate, see section \ref{subsec:Oldroyd}) described in example \ref{exam:Lie-derivative-general-covariance}.

Why is $\overset{\triangledown}{\bk}$ preferred? Just because it is general covariant. However, any other material time derivative, for instance the particle derivative $\dot{\bk}$ of example~\ref{exam:particle-derivative}, could have been chosen to fix an ``origin'' of the affine space of covariant derivatives\footnote{Indeed, this choice seems to have been implicitly adopted by many authors \cite{Hil1978,XBM1998,XBM1998a}.} on $\EE$. Therefore, with our choice of this ``origin'', each covariant derivative on $\EE$ writes
\begin{equation}\label{eq:Olroyd-choice}
  D_{t}\bk = \overset{\triangledown}{\bk} + \Gamma_{\pp}(\pp_{t},\bk),
\end{equation}
where $\bk = \bk_{\tilde{\pp}}$ is a section of the vector bundle $\EE$, defined along the path of embeddings $\tilde{\pp}$, and
\begin{equation*}
  \Gamma_{\pp}: T_{\pp}\Emb(\body,\RR^{3}) \times \EE_{\pp} \to \EE_{\pp}, \qquad (\VV_{\pp},\bk_{\pp}) \mapsto \Gamma_{\pp}(\VV_{\pp},\bk_{\pp})
\end{equation*}
is a bilinear operator.

By the way, the group of diffeomorphisms $\Diff(\espace)$ acts on $\EE$ by the formula
\begin{equation*}
  \varphi \star \bk_{\pp} := \varphi_{*}\bk_{\pp},
\end{equation*}
which is a vector bundle isomorphism, and sends linearly the fiber $\EE_{\pp}$ onto the fiber $\EE_{\varphi \circ \pp}$. The remarkable fact is that, through the second trivialization $\Psi_{2}$, this action reduces to the trivial action on each fiber. Indeed, we have
\begin{equation*}
  \varphi \star (\pp,\bepsilon) := \Psi_{2}(\varphi_{*}\bk_{\pp}) = (\varphi \circ \pp, (\varphi \circ \pp)^{*}(\varphi_{*}\bk_{\pp})) = (\varphi \circ \pp, \bepsilon), \qquad \bepsilon = \pp^{*}\bk_{\pp}.
\end{equation*}

Therefore, thanks to the fact that Oldroyd's rate $\overset{\triangledown}{\bk}$ is general covariant, a covariant derivative $D$ on $\EE$, written as~\eqref{eq:Olroyd-choice}, is general covariant (respectively objective), if and only if
\begin{equation}\label{eq:general-covariance-condition}
  \Gamma_{\varphi(t)\circ\pp(t)}((\varphi(t)\circ\pp(t))_{t}, \bk_{\varphi(t)\circ\pp(t)}) = \varphi(t)_{*}\Gamma_{\pp(t)}(\pp_{t}(t),\bk_{\pp(t)}),
\end{equation}
for any path $\widetilde{\varphi} = (\varphi(t))$ of diffeomorphisms (respectively isometries), any path of embeddings $\widetilde{\pp} = (\pp(t))$, and any path of tensor fields $\bk_{\widetilde{\pp}}$ defined along the path $\widetilde{\pp}$. Using the second trivialization~\eqref{eq:second-triv}, and setting
\begin{equation*}
  \widetilde{\Gamma}_{\pp}(\VV,\bepsilon) := \pp^{*}\Gamma_{\pp}(\VV,\pp_{*}\bepsilon),
\end{equation*}
we can recast condition~\eqref{eq:general-covariance-condition} on $\widetilde{\Gamma}$ as
\begin{equation}\label{eq:Christoffel-covariance}
  \widetilde{\Gamma}_{\varphi(t)\circ\pp(t)}((\varphi(t)\circ\pp(t))_{t},\bepsilon) = \widetilde{\Gamma}_{\pp(t)}(\pp_{t}(t),\bepsilon),
\end{equation}
for any path $\widetilde{\varphi} = (\varphi(t))$ of diffeomorphisms (respectively isometries), any path of embeddings $\widetilde{\pp} = (\pp(t))$, and every tensor fields $\bepsilon \in \Gamma({S^{2}\RR^{3}}^{\star}(\body))$.

Unfortunately, at this level of generality, we have no chance to solve~\eqref{eq:Christoffel-covariance}.
However, if we restrict the class of these operators $\widetilde{\Gamma}_{p}(\VV,\bepsilon)$ to the ones which depend only on $1$-jets of $\pp$ and $\VV$, and on $0$-jets of $\bepsilon$, then we are able to prove a converse of theorem~\ref{thm:objective-derivatives}. We will call such operators \emph{local}, since this terminology is consistent with what is usually called \emph{a local formulation of Continuum Mechanics}. Therefore, we restrict the problem to operators $\widetilde{\Gamma}_{p}(\VV,\bepsilon)$ which are such that
\begin{equation}\label{eq:locality}
  \widetilde{\Gamma}_{\pp}(\VV,\bepsilon)(\XX) = \Upsilon_{(\pp(\XX),\bF(\XX))}\left((\VV(\XX),\bF_{t}(\XX)),\bepsilon(\XX)\right), \qquad \forall \XX \in \body,
\end{equation}
where
\begin{equation*}
  \Upsilon:\: (\RR^{3} \times M_{3}(\RR)) \times (\RR^{3} \times M_{3}(\RR)) \times {S^{2}\RR^{3}}^{\star} \to {S^{2}\RR^{3}}^{\star}
\end{equation*}
is a smooth mapping, which is bilinear in the couple of variables $(\VV(\XX),\bF_{t}(\XX))$ and $\bepsilon(\XX)$. Here, we have implicitly interpreted $\pp$ and $\VV$ as vector valued functions with values in $\RR^{3}$ and used a chart around $\XX \in \body$, so that
\begin{equation*}
  \bF(\XX) := \left(\frac{\partial \pp^{i}}{\partial X^{J}}(\XX)\right), \qquad \bF_{t}(\XX) := \left(\frac{\partial V^{i}}{\partial X^{J}}(\XX)\right),
\end{equation*}
are interpreted as square matrices of size 3. Besides, we denote by $(\cdot)^{s}$ and $(\cdot)^a$, respectively the symmetric and the skew-symmetric parts of a mixed tensor in the Euclidean space $(\RR^{3},\bq)$.

\begin{lem}\label{lem:reduction-lemma}
  Let $D$ be a covariant derivative on $\EE$ as defined by \eqref{eq:Olroyd-choice}, and suppose that $D$ is local, meaning that it satisfies~\eqref{eq:locality}.
  \begin{enumerate}
    \item If $D$ is general covariant, then, $\widetilde{\Gamma}\equiv 0$.
    \item If $D$ is objective, then we have
          \begin{equation}\label{eq:def-upsilon}
            \widetilde{\Gamma}_{\pp}(\VV,\bepsilon)(\XX) = \Upsilon_{\bF(\XX)}\left((\bF_{t}\bF^{-1})^{s}(\XX),\bepsilon(\XX)\right), \qquad \forall \XX \in \body,
          \end{equation}
          where
          \begin{equation}\label{eq:final-objectivity-covariance}
            \Upsilon_{Q\bF(\XX)}\left(Q(\bF_{t}\bF^{-1})^{s}(\XX)Q^{-1},\bepsilon(\XX)\right) = \Upsilon_{\bF(\XX)}\left((\bF_{t}\bF^{-1})^{s}(\XX),\bepsilon(\XX)\right), \qquad \forall Q \in \SO(3).
          \end{equation}
  \end{enumerate}
\end{lem}

\begin{proof}
  The action of a (not necessarily rigid) affine motion $\varphi(t,\xx) = P(t)\xx + \bc(t)$ on $\pp$, $\VV=\pp_{t}$ and $\bepsilon$ translates (pointwise) on the arguments of $\Upsilon$, as follows
  \begin{gather}\label{eq:affine-action}
    \overline{\pp(\XX)} = P\pp(\XX) + \bc, \qquad \overline{\bF(\XX)} = P\bF(\XX), \qquad \overline{\VV(\XX)} = P\VV(\XX) + P_{t}\pp(\XX) + \bc_{t},
    \\
    \overline{\bF_{t}(\XX)} = P\bF_{t}(\XX) + P_{t}\bF(\XX), \qquad \overline{\bepsilon(\XX)} = \bepsilon(\XX),
  \end{gather}
  and the covariance condition~\eqref{eq:Christoffel-covariance}, as
  \begin{multline}\label{eq:Upsilon-covariance}
    \Upsilon_{(P\pp(\XX) + \bc, P\bF(\XX))}\left((P\VV(\XX)+ P_{t}\pp(\XX) + \bc_{t}, P\bF_{t}(\XX) + P_{t}\bF(\XX)), \bepsilon(\XX),\right)
    \\
    = \Upsilon_{(\pp(\XX),\bF(\XX))}\left((\VV(\XX),\bF_{t}(\XX)),\bepsilon(\XX)\right).
  \end{multline}
  We will first make a change of variables and replace the argument $\bF_{t}(\XX)$ of $\Upsilon$ by $\bF_{t}(\XX)\bF(\XX)^{-1}$, with
  \begin{equation*}
    \overline{\bF_{t}(\XX)\bF(\XX)^{-1}} = P(\bF_{t}(\XX)\bF(\XX)^{-1})P^{-1} + \bM,
  \end{equation*}
  where $\bM = P_{t}P^{-1}$. The covariance condition~\eqref{eq:Upsilon-covariance} recasts then as
  \begin{multline}\label{eq:final-Upsilon-covariance}
    \Upsilon_{(P\pp(\XX) + \bc, P\bF(\XX))}\left((P\VV(\XX) + P_{t}\pp(\XX) + \bc_{t}, P(\bF_{t}(\XX)\bF(\XX)^{-1})P^{-1} + \bM),\bepsilon(\XX)\right) =
    \\
    \Upsilon_{(\pp(\XX),\bF(\XX))}\left((\VV(\XX),\bF_{t}(\XX)\bF(\XX)^{-1}),\bepsilon(\XX)\right).
  \end{multline}

  Assume first that $D$ is general covariant. Then, taking $P(t) = I$ and $\bc(t)= t\bc_{0}$ and evaluating~\eqref{eq:final-Upsilon-covariance} at $t=0$, we deduce that $\Upsilon$ does not depend on $\VV(\XX)$. Set now $P(t) = \exp(t\bM)$ and $\bc(t) = 0$, where $\bM$ is a fix square matrix of size 3. We conclude, then, evaluating \eqref{eq:final-Upsilon-covariance} at $t=0$ that $\Upsilon$ does not depend on $(\bF_{t}\bF^{-1})(\XX)$. We have therefore
  \begin{equation*}
    \Upsilon_{(\pp(\XX),\bF(\XX))}\left((\VV(\XX),\bF_{t}(\XX)\bF(\XX)^{-1}),\bepsilon(\XX)\right) = \Upsilon_{(\pp(\XX),\bF(\XX))}\left((0,0),\bepsilon(\XX)\right) = 0,
  \end{equation*}
  and hence, $\widetilde{\Gamma}$ vanishes, which proves point (1).

  Assume now that $D$ is objective and hence that~\eqref{eq:final-Upsilon-covariance} holds only when $P = Q$ is a rotation. Then, the argument above, which shows that $\Upsilon$ does not depend on $\VV(\XX)$, still holds. Taking, now, $Q(t) = \exp(t\bLambda)$, where $\bLambda$ is a skew-symmetric matrix and $\bc(t) = 0$, we conclude, as above, that $\Upsilon$ does not depend on $(\bF_{t}\bF^{-1})^{a}(\XX)$. Finally, taking $Q(t) = I$ and $\bc(t)= \bc_{0}$ in~\eqref{eq:final-Upsilon-covariance}, we deduce that $\Upsilon$ does not depend on $\pp(\XX)$. We conclude that $\Upsilon$ depends only on $\bF(\XX)$, $(\bF_{t}\bF^{-1})^{s}(\XX)$ and $\bepsilon(\XX)$. Therefore, and with a slight abuse of notations, we will write
  \begin{equation*}
    \widetilde{\Gamma}_{\pp}(\VV,\bepsilon)(\XX) = \Upsilon_{\bF(\XX)}\left((\bF_{t}\bF^{-1})^{s}(\XX),\bepsilon(\XX)\right) := \Upsilon_{(0,\bF(\XX))}\left((0,(\bF_{t}(\XX)\bF(\XX)^{-1})^{s}),\bepsilon(\XX)\right),
  \end{equation*}
  where the invariance condition
  \begin{equation}\label{eq:InvQ}
    \Upsilon_{Q\bF(\XX)}\left(Q(\bF_{t}\bF^{-1})^{s}(\XX)Q^{-1},\bepsilon(\XX)\right) = \Upsilon_{\bF(\XX)}\left((\bF_{t}\bF^{-1})^{s}(\XX),\bepsilon(\XX)\right),
  \end{equation}
  holds, for all $Q \in \SO(3)$. This shows point (2) and ends the proof.
\end{proof}

Since its importance, the following obvious corollary of lemma~\ref{lem:reduction-lemma} will be stated as a theorem.

\begin{thm}\label{thm:unicity-general-covariance}
  The Oldroyd objective derivative (Lie derivative) $\overset{\triangledown}{\bk}$ is the unique local material derivative which is general covariant.
\end{thm}

We will now state our expected converse result of theorem~\ref{thm:objective-derivatives} concerning local objective derivatives.

\begin{thm}\label{thm:converse-objectivity}
  Let
  \begin{equation*}
    \dmat{\widetilde{\pp}}{\bk}= \overset{\triangledown}{\bk} + \Gamma_{\pp}(\pp_{t},\bk),
  \end{equation*}
  be a local objective derivative on symmetric second-order covariant tensor fields $\bk$, as defined by~\eqref{eq:objectivity} and \eqref{eq:locality}. Then, there exists a covariant derivative $D$ on $T\Met(\body)$, such that
  \begin{equation*}
    \dmat{\widetilde{\pp}}{\bk} := \widetilde{\pp}_{*}\left(D_{t}(\widetilde{\pp}^{*}\bk)\right).
  \end{equation*}
  In other words, an objective derivative on symmetric second-order covariant tensor fields which depends only on the first jets of $\pp$ and $\pp_{t}=\partial_{t} \pp$ and on zero jets of $\bk$ is equivalent to a covariant derivative on $T\Met(\body)$. Moreover, such a derivative on $T\Met(\body)$ preserves distributions with densities and is thus also equivalent to an objective derivative on symmetric second-order contravariant tensor fields, by the (pseudo) Leibniz rule~\eqref{eq:pseudo-Leibniz-rule}.
\end{thm}

\begin{rem}\label{rem:D-explicit-formula}
  The construction of such a covariant derivative on $T\Met(\body)$, provided in the proof of theorem~\ref{thm:converse-objectivity}, is constructive and requires (in the most general case, see the examples of sections \ref{subsec:Green-Naghdi} and \ref{subsec:Xiao-Bruhns-Meyers}) the choice of a reference configuration $\pp_{0}$. It writes $D_{t} \bepsilon = \partial_{t} \bepsilon + \Gamma^{\pp_{0}}_{\bgamma}(\bgamma_{t}, \bepsilon)$, where
  \begin{equation*}
    \Gamma^{\pp_{0}}_{\bgamma}(\bgamma_{t}, \bepsilon)(\XX) = \Upsilon_{\bF_{0} \bU_{0}(\XX)} \left(\frac{1}{2}\bq^{-1}({\bF_{0}}^{-\star}{\bU_{0}}^{-\star}\bgamma_{t}{\bU_{0}}^{-1}{\bF_{0}}^{-1})(\XX), \bepsilon(\XX)\right),
  \end{equation*}
  $\Upsilon$ is defined by~\eqref{eq:def-upsilon}, $\bF_{0}=T\pp_{0}$, $\bgamma_{0}=\pp_{0}^{*} \, \bq$ and $\bU_{0}=\sqrt{\bgamma_{0}^{-1} \bgamma}$ in the Euclidean space $(T_{X}\body,\bgamma_{0})$. In some cases (see the examples of sections \ref{subsec:Oldroyd} to \ref{subsec:Marsden-Hughes}), however, when the dependance of $\Upsilon_{\bF}$ in $\bF$ is through $\bgamma=\bF^\star \bq \, \bF$, rather than $\bF$, then, the construction of $D$ is straightforward and does not require the particular choice of a reference configuration $\pp_{0}$. These two situations will be illustrated in \autoref{sec:objective-rates}.
\end{rem}

Before providing a proof of theorem~\ref{thm:converse-objectivity}, we will introduce the following lemma which, when only local formulas are involved, justifies why they can be extended from Riemannian metrics which write $\bgamma = \pp^{*}\bq$ to all Riemannian metrics defined on the body.

\begin{lem}\label{lem:extension-lemma}
  Let $\XX_{0} \in \body$. Then, for each metric $\bgamma \in \Met(\body)$ and each $\delta\bgamma \in \Gamma(S^{2}T^{\star}\body)$, there exists a path of embedding $\widetilde{\pp}= (\pp(t))$, such that
  \begin{equation*}
    \bgamma(\XX_{0}) = (\pp^{*}\bq)(0,\XX_{0}), \quad \text{and} \quad \delta\bgamma(\XX_{0}) = (\partial_{t} \pp^{*}\bq)(0,\XX_{0}).
  \end{equation*}
\end{lem}

\begin{proof}
  We need to find a path of embedding $\widetilde{\pp}$ such that, for $t=0$ and $\XX=\XX_{0}$, we have
  \begin{equation}\label{eq:pb-gam-gamt}
    \bgamma(\XX_{0}) = \bF(\XX_{0})^{\star}\bq\, \bF(\XX_{0}), \quad \text{and} \quad \delta\bgamma(\XX_{0}) =  \bF_{t}(\XX_{0})^{\star}\bq \, \bF(\XX_{0}) + \bF(\XX_{0})^{\star}\bq \, \bF_{t}(\XX_{0}),
  \end{equation}
  where $\bF(\XX_{0}) = T_{\XX_{0}}p(0)$ and $\bF_{t}(\XX_{0}) = T_{\XX_{0}}p_{t}(0)$. To do so, fix a reference configuration $\pp_{0}$, and set
  \begin{equation*}
    \bC := \bF_{0}(\XX_{0})^{-\star} \bgamma(\XX_{0}) \bF_{0}(\XX_{0})^{-1}.
  \end{equation*}
  Then, $\bC$ is a positive definite, symmetric covariant tensor and we claim that a solution of the problem~\eqref{eq:pb-gam-gamt} is given by the path $\pp(t) = \bU \exp(t\bA) \pp_{0}$, where $\bU$ is the unique positive square root of the symmetric endomorphism $\bq^{-1}\bC$ and $\bA :=\bC^{-1}\bD$, with
  \begin{equation*}
    \bD := \frac{1}{2}\bF_{0}(\XX_{0})^{-\star} \delta\bgamma(\XX_{0}) \bF_{0}(\XX_{0})^{-1}.
  \end{equation*}
  Indeed, we have at $t=0$
  \begin{equation*}
    \bF(\XX_{0}) = \bU\bF_{0}(\XX_{0}), \qquad \bF_{t}(\XX_{0}) = \bU\bA\bF_{0}(\XX_{0}),
  \end{equation*}
  and one can check that
  \begin{equation*}
    \bF^{\star}(\XX_{0})\bq\, \bF(\XX_{0}) = \bgamma(\XX_{0}), \quad \text{and} \quad \bF_{t}(\XX_{0})^{\star}\bq\, \bF(\XX_{0}) + \bF(\XX_{0})^{\star}\bq\, \bF_{t}(\XX_{0}) = \delta\bgamma(\XX_{0}).
  \end{equation*}
\end{proof}

\begin{proof}[Proof of theorem~\ref{thm:converse-objectivity}]
  The following construction requires the choice of a reference configuration $\pp_{0}$, and we will set $\varphi := \pp \circ \pp_{0}^{-1}$. We will first show that the invariance condition~\eqref{eq:final-objectivity-covariance} on $\Upsilon$ in lemma~\ref{lem:reduction-lemma} allows us to build a well-defined bilinear operator $\Gamma^{\pp_{0}}_{\bgamma}(\delta\bgamma,\bepsilon)$ on $T_{\bgamma}\Met(\body)$ such that
  \begin{equation*}
    \Gamma^{\pp_{0}}_{\widetilde{\pp}^{*}\bq}(\partial_{t}(\widetilde{\pp}^{*}\bq),\bepsilon) = \widetilde{\Gamma}_{\pp}(\partial_{t}\pp, \bepsilon),
  \end{equation*}
  for any path of embeddings $\widetilde{\pp} = (\pp(t))$. To do so, we introduce first the polar decomposition of $\bF_\varphi(\XX)=(\bF \bF_{0}^{-1})(\XX)$, with $\bF_{0}=T\pp_0$, and hence we write
  \begin{equation*}
    \bF(\XX) = \bR(\XX)\bU(\XX)\bF_{0}(\XX),
  \end{equation*}
  where $\bR(\XX)$ is a rotation and $\bU(\XX)$ is the unique positive square root of $\bF_\varphi(\XX)^{t}\bF_\varphi(\XX)$. Next, we introduce $\bU_{0}(\XX) := \bF_{0}(\XX)^{-1}\bU(\XX) \bF_{0}(\XX)$, so that
  \begin{equation*}
    \bF(\XX) = \bR(\XX)\bF_{0}(\XX) \bU_{0}(\XX),
  \end{equation*}
  where $\bU_{0}(\XX)$ is the unique positive square root of the positive symmetric endomorphism (relative to the metric $\bgamma_{0}$)
  \begin{equation*}
    \bgamma_{0}^{-1}(\XX) \bgamma(\XX), \quad \text{where} \quad \bgamma := \pp^{*}\bq, \quad \text{and} \quad \bgamma_{0} := \pp_{0}^{*}\,\bq.
  \end{equation*}
  Now, thanks to theorem~\ref{theo:metric-strain-rate}, we have
  \begin{align*}
    (\bF_{t}\bF^{-1})^{s}(\XX) & = \bq^{-1}\bd(\pp(\XX))
    \\
                               & = \frac{1}{2}\bq^{-1}(\bF^{-\star}\bgamma_{t}\bF^{-1})(\XX)
    \\
                               & = \frac{1}{2}\bq^{-1}(\bR^{-\star}{\bF_{0}}^{-\star}{\bU_{0}}^{-\star}\bgamma_{t}{\bU_{0}}^{-1}{\bF_{0}}^{-1}\bR^{-1})(\XX)
    \\
                               & = \frac{1}{2}(\bR\bq^{-1}{\bF_{0}}^{-\star}{\bU_{0}}^{-\star}\bgamma_{t}{\bU_{0}}^{-1}{\bF_{0}}^{-1}\bR^{-1})(\XX)
  \end{align*}
  and by~\eqref{eq:final-objectivity-covariance}, we get thus (with $Q = \bR(\XX)^{-1}$)
  \begin{equation*}
    \Upsilon_{\bF(\XX) }\left((\bF_{t}\bF^{-1})^{s}(\XX),\bepsilon(\XX)\right) = \Upsilon_{\bF_{0} \bU_{0}(\XX)} \left(\frac{1}{2}\bq^{-1}({\bF_{0}}^{-\star}{\bU_{0}}^{-\star}\bgamma_{t}{\bU_{0}}^{-1}{\bF_{0}}^{-1})(\XX) , \bepsilon(\XX)\right).
  \end{equation*}
  Therefore, since $\bU_{0}(\XX)$ is a function of $\bgamma(\XX)$, this allows us to define a bilinear mapping $\Gamma^{\pp_{0}}_{\bgamma}(\bgamma_{t}, \bepsilon)$, using the formula
  \begin{equation*}
    \Gamma^{\pp_{0}}_{\bgamma}(\bgamma_{t}, \bepsilon)(\XX) := \Upsilon_{\bF_{0} \bU_{0}(\XX)} \left(\frac{1}{2}\bq^{-1}({\bF_{0}}^{-\star}{\bU_{0}}^{-\star}\bgamma_{t}{\bU_{0}}^{-1}{\bF_{0}}^{-1})(\XX), \bepsilon(\XX)\right).
  \end{equation*}
  This mapping $\Gamma^{\pp_{0}}_{\bgamma}(\delta\bgamma, \bepsilon)$ is however defined, \textit{a priori}, only for metrics $\bgamma$ on $\Met(\body)$, which write $\bgamma = \pp^{*}\bq$ and with $\delta\bgamma = \partial_{t}(\pp^{*}\bq)$. By lemma \ref{lem:extension-lemma}, for each metric $\bgamma$ and each variation $\delta\bgamma$, there exists a path of embedding $\widetilde{\pp}$, such that
  \begin{equation*}
    \bgamma(\XX) = (\pp^{*}\bq)(0,\XX), \quad \text{and} \quad \delta\bgamma(\XX) = (\partial_{t} \pp^{*}\bq)(0,\XX).
  \end{equation*}
  This allows us to conclude that this local bilinear mapping $\Gamma^{\pp_{0}}_{\bgamma}(\delta\bgamma, \bepsilon)$ extends and is well defined for all metrics $\bgamma$ and all variations $\delta\bgamma$, and defines a covariant derivative
  \begin{equation*}
    D_{t} \bepsilon := \partial_{t}\bepsilon + \Gamma_{\bgamma}^{\pp_{0}}(\bgamma_{t}, \bepsilon)
  \end{equation*}
  on $T\Met(\body)$. We have moreover by construction
  \begin{equation*}
    \Gamma^{\pp_{0}}_{\pp^{*}\bq}\left(\partial_{t}(\pp^{*}\bq),\bepsilon\right) = \widetilde{\Gamma}_{\pp}(\pp_{t},\bepsilon)(\XX).
  \end{equation*}
  To conclude, observe that the corresponding covariant derivative $D_{t}$ on $T\Met(\body)$ preserves the distributions with densities because it is local (see remark~\ref{rem:local-covariant-derivatives}).
\end{proof}

Finally, we have proved so far that every local objective derivative (on covariant second-order symmetric tensors) derives from a covariant derivative on the space of Riemannian metrics $\Met(\body)$
\begin{equation*}
  D_{t}\bepsilon = \partial_{t}\bepsilon + \Gamma_{\bgamma}(\partial_{t}\bgamma,\bepsilon),
\end{equation*}
where $\Gamma$ is a local bilinear operator which depends smoothly (but in general not linearly on $\bgamma$). It writes thus
\begin{equation*}
  \dmat{\widetilde{\pp}}{\bk_{\widetilde{\pp}}} = \overset{\triangledown}{\bk} - 2\pp_{*}\left(\Gamma_{\pp^{*}\bq}(\pp^{*}\bd,\pp^{*}\bk)\right).
\end{equation*}
There are two special cases which describe all existing objective derivatives of the literature. The first one is defined by those objective derivatives which do not depend explicitly on $\pp^{*}$, which means, more precisely, that there exists a local bilinear operator $B_{\bq}$ such that
\begin{equation*}
  \pp_{*}\Gamma_{\pp^{*}\bq}(\pp^{*}\bd,\pp^{*}\bk) = B_{\bq}(\bd,\bk).
\end{equation*}
Examples of such objective derivatives are provided in subsections \ref{subsec:Zaremba-Jaumann} to \ref{subsec:Marsden-Hughes}. The second one depends of the choice of a reference configuration $\pp_{0}$. We apply first a change of variables and substitute $\widehat{\bgamma} := \bgamma_{0}^{-1}\bgamma$ to $\bgamma$, $\widehat{\bgamma}_{t} :=\bgamma_{0}^{-1}\bgamma_{t}$ to $\bgamma_{t}$ and write
\begin{equation*}
  \widetilde{\Gamma}^{\pp_{0}}_{\widehat{\bgamma}}(\widehat{\bgamma}_{t},\bepsilon) := \Gamma_{\bgamma_{0}\widehat{\bgamma}}(\bgamma_{0}\widehat{\bgamma}_{t},\bepsilon).
\end{equation*}
Then, the second case corresponds to those operators $\widetilde{\Gamma}^{\pp_{0}}$ which do not depend explicitly on $\pp^{*}$, that is when there exists a local bilinear operator $B_{\widehat{\bb}}$ such that
\begin{equation*}
  \pp_{*}\widetilde{\Gamma}^{\pp_{0}}_{\pp^{*}\widehat{\bb}}(\pp^{*}\widehat{\bd},\pp^{*}\bk) = B_{\widehat{\bb}}(\widehat{\bd},\bk).
\end{equation*}
Examples of such objective derivatives are the Green--Naghdi objective rate (\autoref{subsec:Green-Naghdi}) and, more generally, the Xiao--Bruhns--Meyers family (\autoref{subsec:Xiao-Bruhns-Meyers}).

\section{Objective rates in the literature}
\label{sec:objective-rates}

In this section, we show that all objective derivatives found in the literature are induced by some covariant derivative on $T\Met(\body)$,
\begin{equation*}
  D_{t} \bepsilon = \partial_{t}\bepsilon + \Gamma_{\bgamma}(\bgamma_{t}, \bepsilon),
\end{equation*}
and we produce an explicit formula for each of them. Since all of these objective derivatives are local, they induce not only an objective derivative on second-order symmetric \emph{covariant} tensor fields but also one on second-order symmetric \emph{contravariant} tensor fields, using the rule
\begin{equation*}
  D_{t}\btheta = \partial_{t}\btheta - \Gamma^{\star}_{\bgamma}(\bgamma_{t},\btheta),
\end{equation*}
where $\Gamma^{\star}_{\bgamma}(\bgamma_{t},\btheta)$ is the adjoint of the linear mapping defined in remark~\ref{rem:local-covariant-derivatives}. Moreover, these two operators are linked by the Leibniz rule
\begin{equation*}
  D_{t} \btheta : \bepsilon + \btheta : D_{t} \bepsilon = \partial_{t} (\btheta : \bepsilon).
\end{equation*}
Translating this rule on space, rather than on the body, using the change of variables $\bk = \pp_{*}\bepsilon$ and $\btau = \pp_{*}\btheta$, we have accordingly the following (pseudo) Leibniz rule for tensor fields defined on space,
\begin{equation*}
  \dmat{\widetilde{\pp}}{\btau} : \bk + \btau : \dmat{\widetilde{\pp}}{\bk} = \partial_{t}(\btau :\bk) + \Lie_{\uu}(\btau :\bk),
\end{equation*}
where $\uu$ is the Eulerian velocity of the path $\widetilde{\pp}$. Finally, we will use the following notations

\begin{equation*}
  \widehat{\bd} = (\nabla \uu)^{s}=\frac{1}{2}\left( \nabla \uu + (\nabla \uu)^{t} \right)= \bq^{-1}\bd,  \qquad
  \widehat{\bw} = (\nabla \uu)^{a}=\frac{1}{2}\left( \nabla \uu - (\nabla \uu)^{t} \right),
\end{equation*}
and indicate that, in this section, $\btau$ is not related to either the Cauchy or the Kirchhoff stress tensor, but denotes just a contravariant symmetric second-order tensor field defined on $\Omega_{\pp}$.

\subsection{Oldroyd objective rate}
\label{subsec:Oldroyd}

It was introduced in~\cite{Old1950}, and sometimes also referred to as the Lie derivative. It writes
\begin{equation*}
  \overset{\triangledown}{\btau} = \partial_{t}\btau + \Lie_{\uu}\btau = \dot{\btau} - (\nabla \uu) \btau - \btau (\nabla \uu)^{\star},
\end{equation*}
and corresponds (by lemma \ref{lem:fundamental-time-derivative-formula}) to the covariant derivative on $T^{\star}\Met(\body)$ induced by the canonical covariant derivative $D_{t}\bepsilon = \partial_{t}\bepsilon$ on $T\Met(\body)$, which preserves distributions with density, and where we have $D_{t}\btheta = \partial_{t}\btheta$.

\subsection{Truesdell objective rate}
\label{subsec:Truesdell}

It was introduced in~\cite{Tru1955} and corresponds to the variant of Oldroyd's derivative discussed in remark~\ref{rem:Truesdell-trick}. It writes
\begin{equation*}
  \overset{\circ}{\btau}  = \dot{\btau} - (\nabla \uu) \btau - \btau (\nabla \uu)^{\star} + (\dive \uu) \btau = \overset{\triangledown}{\btau} + (\tr\widehat \bd)\, \btau.
\end{equation*}
It corresponds to the following covariant derivative on $T\Met(\body)$
\begin{equation*}
  D_{t}\bepsilon = \partial_{t}\bepsilon - \frac{1}{2} \tr (\bgamma^{-1}\bgamma_{t})\bepsilon ,
\end{equation*}
which is not symmetric.

\subsection{Zaremba--Jaumann objective rate}
\label{subsec:Zaremba-Jaumann}

It was introduced in~\cite{Zar1903,Jau1911} (see also~\cite{Lad1980,Lad1999}) and writes
\begin{equation*}
  \overset{\vartriangle}{\btau} = \dot{\btau} - \widehat{\bw} \btau - \btau \widehat{\bw}^{\star} = \overset{\triangledown}{\btau} + \widehat{\bd} \btau + \btau \widehat{\bd}^{\star}.
\end{equation*}
It was Paul Rougée~\cite{Rou1991,Rou1991a,Rou2006} who realized, for the first time, that this objective rate corresponds to the covariant derivative associated with the metric $G^\mu$~\eqref{eq:Rougee-metric} on $\Met(\body)$.

\begin{thm}[Rougée, 1991]\label{thm:Zaremba-Jaumann}
  The Zaremba--Jaumann derivative corresponds to the covariant derivative on $T\Met(\body)$ given by
  \begin{equation*}
    D_{t} \bepsilon := \partial_{t} \bepsilon - \frac{1}{2} \left(\bgamma_{t} \bgamma^{-1} \bepsilon + \bepsilon \bgamma^{-1}\bgamma_{t}\right).
  \end{equation*}
\end{thm}

\begin{proof}
  From the Leibniz rule, we get immediately
  \begin{equation*}
    D_{t}\btheta := \partial_{t}\btheta + \frac{1}{2} \left( \btheta \bgamma_{t}\bgamma^{-1} + \bgamma^{-1}\bgamma_{t}\btheta \right),
  \end{equation*}
  and thus
  \begin{equation*}
    \pp_{*}\left(D_{t}(\pp^{*}\btau)\right) = \overset{\triangledown}{\btau} + \widehat{\bd} \btau + \btau \widehat{\bd}^{\star},
  \end{equation*}
  because $\pp_{*}\bgamma_{t} = 2\bd$ (theorem~\ref{theo:metric-strain-rate}), $\pp_{*}\bgamma^{-1} = \bq^{-1}$, $\widehat{\bd} = \bq^{-1}\bd$ and $\widehat{\bd}^{\star} = \bd\bq^{-1}$.
\end{proof}

\subsection{Hill objective rates}
\label{subsec:Hill}

Hill~\cite{Hil1978} has introduced the following family of objective derivatives
\begin{equation}\label{eq:Hill-family}
  \begin{aligned}
    \dmat{\widetilde{\pp}}{\btau} & = \dot{\btau} - \left(\widehat{\bw}+m_1 \widehat \bd+m_2\tr( \widehat \bd) \, \Id \right) \btau - \btau \left(\widehat{\bw} + m_1 \widehat \bd+m_2 \tr( \widehat \bd)\, \Id \right)^{\star}
    \\
                                  & = \overset{\triangledown}{\btau} - \left((m_1-1) \widehat \bd + m_2\tr( \widehat \bd) \, \Id \right) \btau - \btau \left((m_1-1)\widehat \bd+m_2 \tr( \widehat \bd)\, \Id \right)^{\star},
  \end{aligned}
\end{equation}
where the two-parameters $m_1$, $m_2$ are real numbers. It contains the Zaremba--Jaumann derivative (for $m_1=m_2=0$), the Oldroyd derivative (for $m_1=1$, $m_2=0$), and the Truesdell derivative (for $m_1=1$ and $m_2=-1/2$). Hill's objective rate corresponds to the following covariant derivative on $T\Met(\body)$,
\begin{equation}\label{eq:Hill-family-D}
  D_{t}\bepsilon = \partial_{t}\bepsilon
  + \frac{1}{2}(m_1-1) \left(\bgamma_{t} \bgamma^{-1} \bepsilon + \bepsilon \bgamma^{-1}\bgamma_{t}\right)
  + m_2 \tr (\bgamma^{-1}\bgamma_{t})\bepsilon.
\end{equation}

\subsection{Fiala objective rate}
\label{subsec:Fiala}

In the same way Rougée~\cite{Rou1991,Rou1991a} derived the Zaremba-Jaumann objective rate from the covariant derivative associated with the metric~\eqref{eq:Rougee-metric}, Fiala proposed in~\cite{Fia2004} a new objective derivative for symmetric second-order covariant tensor fields, which writes
\begin{equation}\label{eq:Fiala-derivative}
  \dmat{\widetilde{\pp}}{\bk} := \dot \bk + \bk \widehat{\bw}+  \widehat{\bw}^{\star} \bk+ \frac{1}{2} \big(
  (\tr \widehat{\bd})\,\bk + \tr(\bq^{-1}\bk)\,\bd - \tr(\widehat{\bd} \bq^{-1} \bk)\, \bq \big).
\end{equation}
It derives from the covariant derivative~\eqref{eq:Ebin-covariant-derivative} associated with Ebin's metric~\eqref{eq:Ebin-metric}. Since this covariant derivative is local, it induces (using pseudo Leibniz rule~\eqref{eq:pseudo-Leibniz-rule}) an objective derivative for symmetric second-order contravariant tensor fields $\btau$, which writes
\begin{equation}\label{eq:Fiala-contravariant}
  \begin{aligned}
    \dmat{\widetilde{\pp}}{\btau} & = \dot{\btau} - \widehat{\bw}\btau - \btau \widehat{\bw}^{\star} + \frac{1}{2}\left(\tr(\btau \bq) \widehat{\bd} \bq^{-1} - \tr(\widehat{\bd})\btau - \tr(\btau \bd)\bq^{-1} \right)
    \\
                                  & = \overset{\triangledown}{\btau} + \widehat{\bd}\btau + \btau \widehat{\bd}^{\star}
    + \frac{1}{2}\left(\tr(\btau \bq) \widehat{\bd} \bq^{-1} - \tr(\widehat{\bd})\btau - \tr(\btau \bd)\bq^{-1} \right)
  \end{aligned}
\end{equation}

\begin{rem}
  By remark~\ref{rem:Truesdell-trick}, and since $\tr(\widehat \bd)=\dive \uu$, we define another objective rate by
  \begin{equation}\label{eq:Fiala-contravariant-Truesdell}
    \frac{\bar{d}_{\widetilde{\pp}}\btau}{dt} := \rho \, \dmat{\widetilde{\pp}}{}\left(\frac{\btau}{\rho}\right) = \dot \btau
    - \widehat{\bw}\btau - \btau \widehat{\bw}^{\star} + \frac{1}{2}\left(\tr(\bq \btau) \widehat{\bd}\bq^{-1} + \tr(\widehat{\bd})\btau -\tr( \btau \bd)\bq^{-1} \right),
  \end{equation}
  to which corresponds the covariant derivative~\eqref{eq:Ebin-covariant-derivative} with the additional term $-\frac{1}{2} \tr (\bgamma^{-1}\bgamma_{t})\bepsilon$.
\end{rem}

\subsection{Marsden--Hughes objective rates}
\label{subsec:Marsden-Hughes}

In~\cite[Chapter 1, Box 6.1]{MH1994}, Marsden and Hughes claimed that all objective rates of second-order tensor fields are in fact Lie derivatives. More precisely, they defined the following basic objective rates\footnote{Note that what are denoted by $\mathcal L_{\vv}\bsigma^{2}$ and $\mathcal L_{\vv}\bsigma^{3}$ in~\cite[Chapter 1, Box 6.1]{MH1994} are not symmetric second-order tensors, and a mean of them is required in order to build an objective derivative on symmetric contravariant second-order tensors. This is our second objective derivative $d^{2}$.} and asserted that each objective rate of a contravariant symmetric second-order tensor is a linear combination of such derivatives and of the variant explained in remark~\ref{rem:Truesdell-trick},
\begin{align*}
  \frac{d^{1}_{\widetilde{\pp}}\btau}{dt} & := \partial_{t}\btau + \Lie_{\uu}\btau = \dot{\btau} - (\nabla \uu) \btau - \btau (\nabla \uu)^{\star},
  \\
  \frac{d^{2}_{\widetilde{\pp}}\btau}{dt} & := \partial_{t}\btau + \frac{1}{2}\left\{\Lie_{\uu}(\btau\bq)\bq^{-1} + \bq^{-1}\Lie_{\uu}(\bq\btau)\right\} = \dot{\btau} - \widehat{\bw} \btau - \btau \widehat{\bw}^{\star},
  \\
  \frac{d^{3}_{\widetilde{\pp}}\btau}{dt} & := \partial_{t}\btau + \bq^{-1}\Lie_{\uu}(\bq\btau\bq)\bq^{-1} = \dot{\btau} + (\nabla \uu)^{t}\btau + \btau((\nabla \uu)^{t})^{\star},
  \\
  \frac{d^{4}_{\widetilde{\pp}}\btau}{dt} & := \rho \frac{d^{1}_{\widetilde{\pp}}}{dt}\left(\frac{\btau}{\rho}\right) = \dot{\btau} - (\nabla \uu) \btau - \btau (\nabla \uu)^{\star} + (\dive \uu) \btau,
\end{align*}
where $\rho$ is the mass density. These calculations are done using the formulas of proposition~\ref{prop:Lie-derivative-versus-covariant-derivative} and the fact that a covariant or a contravariant second-order tensor $\bt$ is symmetric if and only if $\bt^{\star} = \bt$.

\begin{rem}
  Note that the first objective derivative $d^{1}$ is just the Oldroyd derivative (\autoref{subsec:Oldroyd}), the second one $d^{2}$ is the Zaremba--Jaumann derivative (\autoref{subsec:Zaremba-Jaumann}) and the fourth one $d^{4}$ is the Truesdell derivative (\autoref{subsec:Truesdell}).
\end{rem}

These four objective derivatives correspond respectively to the following covariant derivatives on $T\Met(\body)$:
\begin{align*}
  D_{t}^{1}\bepsilon & = \partial_{t}\bepsilon,
  \\
  D_{t}^{2}\bepsilon & = \partial_{t}\bepsilon - \frac{1}{2} \left(\bgamma_{t} \bgamma^{-1} \bepsilon + \bepsilon \bgamma^{-1}\bgamma_{t}\right),
  \\
  D_{t}^{3}\bepsilon & = \partial_{t}\bepsilon - \left(\bgamma_{t} \bgamma^{-1} \bepsilon + \bepsilon \bgamma^{-1}\bgamma_{t}\right),
  \\
  D_{t}^{4}\bepsilon & = \partial_{t}\bepsilon - \frac{1}{2} \tr (\bgamma^{-1}\bgamma_{t})\bepsilon.
\end{align*}
These four expressions are not linearly independent. A covariant derivative $D_{t}\bepsilon = \partial_{t}\bepsilon + \Gamma_{\bgamma}(\bgamma_{t},\bepsilon)$ on $T\Met(\body)$ is a linear combination of $D_{t}^{1}\bepsilon$, $D_{t}^{2}\bepsilon$, $D_{t}^{3}\bepsilon$, $D_{t}^{4}\bepsilon$, if and only if there exists real constants $\alpha$ and $\beta$ such that
\begin{equation*}
  \Gamma_{\bgamma}(\bgamma_{t},\bepsilon) = \alpha \left(\bgamma_{t} \bgamma^{-1} \bepsilon + \bepsilon \bgamma^{-1}\bgamma_{t}\right) + \beta \tr (\bgamma^{-1}\bgamma_{t})\bepsilon.
\end{equation*}

\begin{rem}
  One will note that Marsden--Hughes objective rates family coincides with Hill's family~\eqref{eq:Hill-family}--\eqref{eq:Hill-family-D}.
\end{rem}

Furthermore, it is clear that Fiala's objective derivative (\autoref{subsec:Fiala}), which corresponds to the covariant derivative~\eqref{eq:Ebin-covariant-derivative}, cannot be written this way. Marsden and Hughes' claim is therefore false. Many more objective rates can be built which are not Lie derivatives, for instance, all those which depend explicitly on the choice of a reference configuration $\pp_{0}$.

\subsection{Green--Naghdi objective rate}
\label{subsec:Green-Naghdi}

It was introduced in~\cite{GN1965}, and defined using a reference configuration $\pp_{0}: \body \to \Omega_{0}$. Introducing the deformation $\varphi := \pp \circ \pp_{0}^{-1}$, we write $\bF_{\varphi} = T\varphi = \bR \bU$, where $\bR$ is a rotation and $\bU$ is the unique positive square root of the mixed right Cauchy--Green tensor $\bq^{-1} \bC$, where $\bC = \bF_{\varphi}^{\star}\bq\, \bF_{\varphi}$. Then the Green-Naghdi objective rate is defined as
\begin{equation}\label{eq:Green-Naghdi}
  \overset{\Box}{\btau} := \dot{\btau} - \btau \bomega^{\star} - \bomega \btau, \qquad \bomega := \bR_{t}\bR^{-1}.
\end{equation}
This objective derivative is local and derives thus from a covariant derivative on $T\Met(\body)$, according to theorem~\ref{thm:converse-objectivity}. The derivation of this covariant derivative follows the lines of the proof of theorem~\ref{thm:converse-objectivity} and illustrates its effectiveness.

\begin{thm}
  The Green-Naghdi derivative~\eqref{eq:Green-Naghdi} corresponds to the covariant derivative
  \begin{equation*}
    D_{t} \bepsilon := \partial_{t} \bepsilon - \bepsilon \left(\bU_{0}^{-1} {\bL_{\bU_{0}}}^{-1}(\bgamma_{0}^{-1}\bgamma_{t})\right) - \left(\bU_{0}^{-1} {\bL_{\bU_{0}}}^{-1}(\bgamma_{0}^{-1}\bgamma_{t})\right)^{\star}\bepsilon
  \end{equation*}
  on $T\Met(\body)$, where $\bU_{0}(\XX)$ is the unique positive square root of the positive symmetric endomorphism $\bgamma_{0}^{-1}\bgamma$ of the Euclidean space $(T_{\XX}\body, \bgamma_{0}(\XX))$ and
  \begin{equation*}
    \bL_{\bU_{0}}: \: \End_{s}(T_{\XX}\body) \to \End_{s}(T_{\XX}\body), \qquad \bS \mapsto \bU_{0}\bS + \bS\bU_{0}.
  \end{equation*}
\end{thm}

\begin{proof}
  First, we use the pseudo Leibniz rule~\eqref{eq:pseudo-Leibniz-rule}, to derive the corresponding objective derivative on second-order covariant vector fields
  \begin{equation*}
    \overset{\Box}{\bk} := \dot{\bk} + \bomega^{\star}\bk + \bk \bomega = \overset{\triangledown}{\bk}  + (\bomega -\nabla \uu)^{\star}\bk + \bk(\bomega -\nabla \uu).
  \end{equation*}
  Next, using the notations of \autoref{sec:converse-theorem}, we deduce that
  \begin{equation*}
    \widetilde{\Gamma}_{\pp}(\bV,\bepsilon) = (\pp^{*}(\bomega -\nabla \uu))^{\star}\bepsilon + \bepsilon\pp^{*}(\bomega -\nabla \uu), \qquad \uu = \bV \circ \pp,
  \end{equation*}
  and we have to express $\pp^{*}(\bomega -\nabla \uu)(\XX)$ as a function of $\bgamma(\XX)$ and $\bgamma_{t}(\XX)$ for all $\XX\in \body$, where $\bgamma = \pp^{*}\bq$ is the pull-back of the Euclidean metric $\bq$. To do so, observe that
  \begin{equation*}
    \nabla \uu = (\bF_{\varphi})_{t}\bF_{\varphi}^{-1} = (\bR_{t}\bU + \bR\bU_{t})(\bR\bU)^{-1} = \bR_{t}\bR^{-1} + \bF_{\varphi}(\bU^{-1}\bU_{t})\bF_{\varphi}^{-1} = \bomega + \varphi_{*}(\bU^{-1}\bU_{t}),
  \end{equation*}
  and thus that
  \begin{equation}\label{eq:pourGN}
    \bomega - \nabla \uu = - \varphi_{*}(\bU^{-1}\bU_{t}).
  \end{equation}
  Now, introducing the linear tangent map $\bF_{0} = Tp_{0}$ and $\bU_{0}=p_{0}^{*}\bU$, we have ${\bU_{0}}_{t} = p_{0}^{*} \bU_{t}$, and
  \begin{equation*}
    \pp^{*}(\bomega - \nabla \uu) = \pp_{0}^{*}\varphi^{*}(\bomega - \nabla \uu) = - \pp_{0}^{*}(\bU^{-1}\bU_{t}) = -\bU_{0}^{-1}{\bU_{0}}_{t}.
  \end{equation*}
  We have therefore
  \begin{equation*}
    \widetilde{\Gamma}_{\pp}(\bV,\bepsilon) = - \bepsilon\left(\bU_{0}^{-1} {\bU_{0}}_{t}\right) - \left(\bU_{0}^{-1} {\bU_{0}}_{t}\right)^{\star}\bepsilon,
  \end{equation*}
  where $\bU_{0}^{2}= \bgamma_{0}^{-1}\bgamma$ and $\bU_{0}\,{\bU}_{0t} + {\bU}_{0t}\,\bU_{0} = \bgamma_{0}^{-1}\bgamma_{t}$. Finally, we use the fact that the linear mapping
  \begin{equation}
    \bL_{P} : \mathrm{End}_{s}(E) \to \mathrm{End}_{s}(E), \qquad S \mapsto PS + SP,
  \end{equation}
  defined on the space $\mathrm{End}_{s}(E)$ of symmetric endomorphisms of an Euclidean space $E$ is invertible if $P$ is positive definite. We conclude that $\widetilde{\Gamma}_{\pp}$ induces the following well-defined Christoffel operator on $T\Met(\body)$
  \begin{equation*}
    \Gamma_{\bgamma}^{\pp_{0}}(\bgamma_{t}, \bepsilon) := - \bepsilon \left(\bU_{0}^{-1} {\bL_{\bU_{0}}}^{-1}(\bgamma_{0}^{-1}\bgamma_{t})\right) - \left(\bU_{0}^{-1} {\bL_{\bU_{0}}}^{-1}(\bgamma_{0}^{-1}\bgamma_{t})\right)^{\star}\bepsilon.
  \end{equation*}
\end{proof}

\begin{rem}
  The Green--Naghdi is another example of an objective derivative which is not a Lie derivative as defined in~\autoref{subsec:Marsden-Hughes}, since it depends explicitly on a reference configuration $\pp_{0}$ (through $\bgamma_{0}=\pp^{*} \bq$).
\end{rem}

\subsection{Xiao--Bruhns--Meyers objective rates}
\label{subsec:Xiao-Bruhns-Meyers}

A general family of co-rotational objective derivatives, extending Hill's family,  has been obtained by Xiao and coworkers in~\cite{XBM1998},
\begin{equation}\label{eq:corot}
  \dmat{\widetilde{\pp}}{\btau} = \dot{\btau} - \widehat{\bOmega} \btau - \btau \widehat{\bOmega}^{\star}, \qquad \widehat \bOmega = \widehat{\bw} +  \widehat{\bUpsilon} \big(\widehat{\bb}, \widehat{\bd}\big),
\end{equation}
which recast as
\begin{equation*}
  \dmat{\widetilde{\pp}}{\btau} = \overset{\triangledown}{\btau} - \left( \widehat{\bUpsilon} \big(\widehat{\bb}, \widehat{\bd}\big)-\widehat{\bd} \right)\btau - \btau \left( \widehat{\bUpsilon} \big(\widehat{\bb}, \widehat{\bd}\big)-\widehat{\bd} \right)^{\star},
\end{equation*}
where
\begin{equation}\label{eq:UpsilonXiao}
  \widehat{\bUpsilon} \big(\widehat{\bb}, \widehat{\bd}\big) := \nu_{1}(\widehat{\bb}) \left(\widehat{\bb} \widehat{\bd}\right)^{a} + \nu_{2}(\widehat{\bb}) \left(\widehat{\bb}^2 \widehat{\bd}\right)^{a} + \nu_{3}(\widehat{\bb}) \left(\widehat{\bb}\widehat{\bd}\widehat{\bb}^2\right)^{a},
\end{equation}
$\nu_{k}(\widehat{\bb})$ are functions of the fundamental isotropic invariants of $\widehat{\bb}$, and $\widehat{\bb} = \bb \bq=(\varphi_{*} \bq^{-1})\bq$ is the left Cauchy--Green mixed tensor~\eqref{eq:left-Cauchy-Green-tensor}.

\begin{rem}
  This family contains the Jaumann derivative, the Oldroyd--Lie derivative and the Green-Naghdi derivative. It does not, however, contains the objective derivatives~\eqref{eq:Fiala-contravariant}-\eqref{eq:Fiala-contravariant-Truesdell} of Fiala type.
\end{rem}

This family of objective derivatives is local and induces thus, by the (pseudo) Leibniz rule, the following objective derivatives on second-order covariant tensor fields $\bk$
\begin{equation}\label{eq:corot-k}
  \dmat{\widetilde{\pp}}{\bk} = \dot{\bk} + \bk\, \widehat{\bOmega} + \widehat{\bOmega}^{\star} \bk
  = \overset{\triangledown}{\bk} + \bk\, \left( \widehat{\bUpsilon} \big(\widehat{\bb}, \widehat{\bd}\big)-\widehat{\bd} \right) +  \left( \widehat{\bUpsilon} \big(\widehat{\bb}, \widehat{\bd}\big)-\widehat{\bd} \right)^{\star} \bk.
\end{equation}

\begin{thm}
  The choice of a reference configuration $\pp_{0}$ allows us to recast the Xiao-Bruhns-Meyers objective derivative~\eqref{eq:corot-k} as a covariant derivative on $T\Met(\body)$
  \begin{equation*}
    D_{t} \bepsilon := \partial_{t} \bepsilon + \Gamma^{\pp_{0}}_{\bgamma}(\bgamma_{t},\bepsilon),
  \end{equation*}
  where
  \begin{equation*}
    \Gamma^{\pp_{0}}_{\bgamma}(\bgamma_{t},\bepsilon) := - \frac{1}{2} \left(\bgamma_{t} \bgamma^{-1} \bepsilon + \bepsilon \bgamma^{-1}\bgamma_{t}\right) + \bepsilon L_{\bgamma_{0}^{-1}\bgamma}(\bgamma_{0}^{-1} \bgamma_{t}) + \left(L_{\bgamma_{0}^{-1}\bgamma}(\bgamma_{0}^{-1} \bgamma_{t})\right)^{\star}\bepsilon,
  \end{equation*}
  and
  \begin{equation*}
    L_{\bgamma_{0}^{-1}\bgamma}(\bgamma_{0}^{-1} \bgamma_{t}) := \pp^{*}\widehat{\bUpsilon} \big(\pp_{*}(\bgamma_{0}^{-1}\bgamma) , \frac{1}{2} \pp_{*}(\bgamma^{-1} \bgamma_{t})\big).
  \end{equation*}
\end{thm}

\begin{proof}
  We use the fact (see remark~\ref{rem:CG-pullbacks}) that $\pp^{*}\bd = \frac{1}{2} \bgamma_{t}$ and $ \pp^{*}\bb = \bgamma_{0}^{-1}$,
  from which we get
  \begin{equation*}
    p^{*}(\widehat{\bd}) = p^{*}(\bq^{-1}\bd) = \frac{1}{2}\bgamma^{-1}\bgamma_{t}, \qquad p^{*} (\widehat{\bb} ) = p^{*}(\bb\bq) = \bgamma_{0}^{-1}\bgamma.
  \end{equation*}
  Besides, $\nu_{k}(\widehat{\bb})$ is a function $f_{k}$ of $\tr \widehat{\bb}$, $\tr(\widehat{\bb}^{2})$, $\tr(\widehat{\bb}^{3})$ so that
  \begin{equation*}
    p^{*} \left(\nu_{k}(\widehat{\bb})\right) = f_{k}(\tr(p^{*}\widehat{\bb}),\tr (p^{*}\widehat{\bb}^{2}),\tr p^{*} (\widehat{\bb}^{3}))\circ \pp
  \end{equation*}
  Therefore, by \eqref{eq:UpsilonXiao}, there exists a local linear operator $L_{\bgamma_{0}^{-1}\bgamma}$, depending smoothly on $\bgamma_{0}^{-1}\bgamma$ such that
  \begin{equation*}
    L_{\bgamma_{0}^{-1}\bgamma}(\bgamma_{0}^{-1} \bgamma_{t}) := \pp^{*} \widehat{\bUpsilon} \big(\pp_{*}(\bgamma_{0}^{-1}\bgamma), \frac{1}{2}\pp_{*}(\bgamma^{-1}\bgamma_{t})\big),
  \end{equation*}
  and using~\eqref{eq:corot-k}, we get
  \begin{equation*}
    p^{*}\left(\dmat{\widetilde{\pp}}{\bk} \right) =D_{t} \bepsilon = \partial_{t}\bepsilon + \Gamma^{\pp_{0}}_{\bgamma}(\bepsilon, \bgamma_{t}),
  \end{equation*}
  where we have set
  \begin{equation*}
    \Gamma^{\pp_{0}}_{\bgamma}(\bgamma_{t},\bepsilon) := - \frac{1}{2} \left(\bgamma_{t} \bgamma^{-1} \bepsilon + \bepsilon\, \bgamma^{-1}\bgamma_{t}\right) + \bepsilon L_{\bgamma_{0}^{-1}\bgamma}(\bgamma_{0}^{-1} \bgamma_{t}) + \left(L_{\bgamma_{0}^{-1}\bgamma}(\bgamma_{0}^{-1} \bgamma_{t})\right)^{\star}\bepsilon.
  \end{equation*}
\end{proof}

\section{Conclusion}
\label{sec:conclusion}

We have enhanced the geometrical framework of Continuum Mechanics, insisting on the role of $\Met(\body)$, the manifold of Riemannian metrics on $\body$. Working on the body $\body$, an abstract manifold with boundary instead of a configuration embedded in Euclidean space, has forced us to take some care when developing the finite strain theory. We have been led to properly define objectivity/material frame indifference of tensors in a rigorous geometric manner as an equivariant section of a certain (infinite dimensional) vector bundle. In this geometric framework, the objectivity of elasticity laws (in the sense of Rougée) becomes a simple theorem and the objective derivatives/rates are interpreted as covariant derivatives on $\Met(\body)$.

As stated in the introduction, the question -- clearly formulated by Marsden and Hughes in~\cite{MH1994} -- of the nature of the objective derivatives is not new. It has been initially addressed by classifying them into two categories: ``Lie type'' and ``co-rotational objective rates'' (see \cite{XBM1998a} for a review). In the second case, a natural Leibniz rule is exhibited, which was proved useful for inelasticity thermodynamic formulation \cite{Lad1980,Lad1999}. Xiao, Bruhns and Meyers \cite{XBM1998} have then derived a rather general expression for co-rotational objective rates. Rougée \cite{Rou1991,Rou1991a,Rou2006} has observed that the Jaumann rate was in fact a covariant derivative on $\Met(\body)$, the manifold of Riemannian metrics on the body, and Fiala provided thereafter a new objective rate that expresses the same way in~\cite{Fia2004}. This observation finally proves to be general. Indeed, by theorem~\ref{thm:objective-derivatives}, each covariant derivative on $\Met(\body)$ induces an objective derivative on covariant symmetric second-order tensor fields, and also on contravariant symmetric second-order tensor fields if it preserves distributions with densities. Conversely, by theorem \ref{thm:converse-objectivity}, each objective derivative that depends only on first jets of the embedding $\pp$ and of the Lagrangian velocity $\VV$ (\emph{i.e.} that are considered as \emph{local} from Continuum Mechanics point of view) recasts as a covariant derivative on $\Met(\body)$. Meanwhile, we have shown that the Oldroyd--Lie rate is the only local objective derivative which is general covariant.

Finally, we have illustrated our work by calculating the expression of the corresponding covariant derivative on $\Met(\body)$ for all objective derivatives found in the literature, in particular Hill's and Xiao--Bruhns--Meyers' families. Some of them require the choice of a reference configuration and others not. We have furthermore observed that Marsden and Hughes' claim that all the objective rates are linear combinations of Lie derivatives is false.

\appendix

\section{Second-order tensors and their interpretations}
\label{sec:second-order-tensors}

Given a real (finite dimensional) vector space $E$ and denoting its dual by $E^{\star}$, we can define four types of second-order tensors on $E$, which we interpret as bilinear mappings.
\begin{enumerate}
  \item $\bb: E \times E \to \RR$, $(\xx,\yy) \mapsto \bb(\xx,\yy)$,
  \item $\bb: E \times E^{\star} \to \RR$, $(\xx,\beta) \mapsto \bb(\xx,\beta)$,
  \item $\bb: E^{\star} \times E \to \RR$, $(\alpha,\yy) \mapsto \bb(\alpha,\yy)$,
  \item $\bb: E^{\star} \times E^{\star} \to \RR$, $(\alpha,\beta) \mapsto \bb(\alpha,\beta)$.
\end{enumerate}
To each of these tensors, we associate, using the convention, of that we call \emph{the second argument}, a linear mapping
\begin{enumerate}
  \item $\tilde{\bb}: E \to E^{\star}$, $\yy \mapsto \bb(\cdot,\yy)$,
  \item $\tilde{\bb}: E^{\star} \to E^{\star}$, $\beta \mapsto \bb(\cdot,\beta)$,
  \item $\tilde{\bb}: E \to (E^{\star})^{\star} = E$, $\yy \mapsto \bb(\cdot,\yy)$,
  \item $\tilde{\bb}: E^{\star} \to (E^{\star})^{\star} = E$, $\beta \mapsto \bb(\cdot,\beta)$,
\end{enumerate}
and if there is no ambiguity, we will not distinguish between $\bb$ and $\tilde{\bb}$ and only use the notation $\bb$. Given a basis $(\ee_{i})$ of $E$, and denoting its dual basis by $(\ee^{i})$, their respective components write $(\tensor{b}{_{ij}})$, $(\tensor{b}{_{i}}{^{j}})$, $(\tensor{b}{^{i}}{_{j}})$ and $(\tensor{b}{^{ij}})$.

Given two vector spaces $E$ and $F$, and a linear mapping $L : E \to F$, its \emph{adjoint} (or \emph{dual linear mapping}) is defined as
\begin{equation*}
  L^{\star}: F^{\star} \to E^{\star}, \qquad \alpha \mapsto \alpha \circ L.
\end{equation*}
Note that $(L^{-1})^{\star} = (L^{\star})^{-1}$ and $(L_{1} \, L_{2})^{\star} = L_{2}^{\star}\, L_{1}^{\star}$.

\begin{rem}
  A second-order covariant, or contravariant, tensor $\bb$ is symmetric if and only if $\bb^{\star} = \bb$.
\end{rem}

When the spaces $E$ and $F$ are respectively equipped with scalar products, noted $\bq_{E}$ and $\bq_{F}$ respectively, the \emph{transpose} $L^{t}: F \to E$ of a linear mapping $L: E \to F$ is defined implicitly by the relation
\begin{equation*}
  \langle L\xx, \yy \rangle_{F} = \langle \xx, L^{t}\yy \rangle_{E}.
\end{equation*}
The following diagram makes clear the relation between $L^{\star}$ and $L^{t}$
\begin{equation*}
  \xymatrix{
  E^{\star}   & F^{\star} \ar[l]_{L^{\star}} \\
  E \ar[u]^{\bq_{E}} \ar[r]^{L}  & F \ar@/^1pc/[l]^{L^{t}} \ar[u]_{\bq_{F}} }
\end{equation*}
and leads to
\begin{equation*}
  \bq_{E}\,L^{t} = L^{\star}\,\bq_{F}.
\end{equation*}

\section{Pull-back and push-forward}
\label{sec:pullback-pushforward}

The fundamental concept of differential geometry that allows to pass from material variables to spatial variables (and vice versa) are the operations of \emph{pull-back} and \emph{push-forward}. For functions, these operations are defined by
\begin{equation*}
  \pp^{*}f = f \circ \pp \quad \text{(pull-back)}, \qquad \pp_{*}F = F \circ \pp^{-1} \quad \text{(push-forward)},
\end{equation*}
where $f \in \Cinf(\Omega,\RR)$ and $F \in \Cinf(\body,\RR)$. For \emph{vector fields}, the following diagram
\begin{equation*}
  \xymatrix{
  T\body \ar[r]^{T\pp} \ar[d]^{\pi}   & T\espace \ar[d]^{\pi} \\
  \body \ar@/^1pc/[u]^{\UU} \ar[r]^{\pp}  & \espace \ar@/_1pc/[u]_{\uu} }
\end{equation*}
leads immediately to the natural definitions
\begin{equation*}
  \pp^{*}\uu = T\pp^{-1} \circ \uu \circ \pp \quad \text{(pull-back)}, \qquad \pp_{*}\UU = T\pp \circ \UU \circ \pp^{-1} \quad \text{(push-forward)}.
\end{equation*}
For \emph{covector fields}, the following diagram
\begin{equation*}
  \xymatrix{
  T^{\star}\body \ar[d]^{\pi}   & T^{\star}\espace \ar[d]^{\pi} \ar[l]_{T\pp^{\star}} \\
  \body \ar@/^1pc/[u]^{\alpha} \ar[r]^{\pp}  & \espace \ar@/_1pc/[u]_{\beta} }
\end{equation*}
leads to the following definitions
\begin{equation*}
  \pp^{*}\beta = T\pp^{\star} \circ \beta \circ \pp \quad \text{(pull-back)}, \qquad \pp_{*}\alpha = (T\pp^{\star})^{-1} \circ \alpha \circ \pp^{-1} \quad \text{(push-forward)}.
\end{equation*}

The pull-back and push-forward operations are inverse to each other, meaning that $\pp^{*} = (\pp_{*})^{-1} = (\pp^{-1})_{*}$. They are easily extended to higher-order contravariant, covariant or mixed tensor fields.

In local coordinate systems, $(X^{I})$ on $\body$ and $(x^{i})$ on $\espace$, where we have set $\pp(\XX) = \xx$, the linear tangent map $T\pp : T\body \to T\espace$ is represented by the square matrix $\bF$ defined by
\begin{equation*}
  \tensor{F}{^{i}}{_{J}} = \frac{\partial x^{i}}{\partial X^{J}},
\end{equation*}
and its dual tangent linear map $T\pp^{\star} : T^{\star}_{\xx}\espace \to T^{\star}_{\XX}\body$, by
\begin{equation*}
  \tensor{(\bF^{\star})}{_{I}}{^{j}} = \frac{\partial x^{j}}{\partial X^{I}}.
\end{equation*}
We will write $\bF^{-\star} := (\bF^{\star})^{-1}=(\bF^{-1})^{\star}$.

\begin{prop}\label{prop:pb-pf-order-1-and-2}
  For tensor fields of order one or two, we have the following expressions.
  \begin{enumerate}
    \item For \emph{contravariant vector fields} $\WW=(W^{I})$, $\ww=(w^{i})$ , we have
          \begin{equation*}
            \pp_{*} \WW = \bF (\WW \circ \pp^{-1}) ,
            \qquad
            \pp^{*} \ww = \bF^{-1} (\ww \circ \pp).
          \end{equation*}
    \item For \emph{covariant vector fields} $\alpha=(\alpha_{I})$, $\beta=(\beta_{i})$, we have
          \begin{equation*}
            \pp_{*} \alpha = \bF^{-\star}(\alpha \circ \pp^{-1}),
            \qquad
            \pp^{*} \beta = \bF^{\star}(\beta \circ \pp).
          \end{equation*}
    \item For \emph{second-order covariant tensor fields} $\bepsilon= (\varepsilon_{IJ})$,  $\bk = (k_{ij})$, we have
          \begin{equation*}
            \pp_{*} \bepsilon = \bF^{-\star} (\bepsilon \circ \pp^{-1}) \bF^{-1},
            \qquad
            \pp^{*} \bk = \bF^{\star} (\bk \circ \pp)  \bF .
          \end{equation*}
    \item For \emph{second-order contravariant tensor fields} $\btheta = (\btheta^{IJ})$ and $\btau = (\btau^{ij})$, we have
          \begin{equation*}
            \pp_{*} \btheta = \bF (\btheta \circ \pp^{-1})\bF^{\star} ,
            \qquad
            \pp^{*} \btau = \bF^{-1} (\btau  \circ \pp) \bF^{-\star}.
          \end{equation*}
    \item For \emph{second-order mixed tensor fields} $\hat{\bT} = ({\hat{T}^{I}}_{\; J})$ and $\hat{\bt} = ({\hat{t}^{i}}_{\, j})$, we have
          \begin{equation*}
            \pp_{*} \hat{\bT} = \bF (\hat{\bT} \circ \pp^{-1})\bF^{-1} ,
            \qquad
            \pp^{*} \hat{\bt} = \bF^{-1} (\hat{\bt} \circ \pp) \bF  .
          \end{equation*}
    \item
          For \emph{second-order mixed tensor fields} $\check{\bT} = ({\check{T}_{I}}^{\,J})$ and $\check{\bt} = ({\check{t}_{i}}^{\; j})$, we have
          \begin{equation*}
            \pp_{*} \check{\bT} = \bF^{-\star} (\check{\bT} \circ \pp^{-1}) \bF^{\star} ,
            \qquad
            \pp^{*} \check{\bt} = \bF^{\star} (\check{\bt} \circ \pp)\bF^{-\star}  .
          \end{equation*}
  \end{enumerate}
\end{prop}

\begin{rem}
  The \emph{push-forward} and \emph{pull-back} operations commute with the contraction between covariant and contravariant tensors. So in particular, we get
  \begin{equation*}
    (\pp_{*}\alpha) \cdot (\pp_{*}\WW) = \pp_{*}(\alpha \cdot \WW), \qquad (\pp_{*} \btheta) : (\pp_{*} \bepsilon) = \pp_{*} (\btheta: \bepsilon), \qquad \tr(\pp_{*} \hat{\bT}) = \pp_{*}(\tr \hat{\bT}).
  \end{equation*}
\end{rem}

\section{Lie derivative}
\label{sec:Lie-derivative}

Given a manifold $M$, the \emph{Lie derivative} of a (time-independent) tensor field $\bt \in \mathbb{T}(M)$ corresponds to the infinitesimal version of the pull-back operation. More precisely, let $\uu$ be a vector field on $M$, $\phi(t)$ its flow and $\bt$ be a tensor field on $M$. The \emph{Lie derivative} $\Lie_{\uu} \bt$ of the tensor field $\bt$ with respect to $\uu$ is defined by
\begin{equation*}
  \Lie_{\uu} \bt := \left.\frac{\partial}{\partial t}\right|_{t=0} \phi(t)^{*} \bt .
\end{equation*}
The Lie derivative has the following properties.
\begin{enumerate}
  \item When $\bt = \vv$ is a vector field, $\Lie_{\uu} \vv = [\uu,\vv]$, where $[\uu,\vv]$ is the \emph{Lie bracket} of the two vector fields $\uu$ and $\vv$.
  \item Given a diffeomorphism $\varphi$ of $M$, we have
        \begin{equation}\label{eq:pullback-and-Lie-derivative}
          \varphi^{*} \Lie_{\uu} \bt = \Lie_{\varphi^{*}\uu} \varphi^{*}\bt.
        \end{equation}
  \item At any time $t$ where $\phi(t)$ is defined, we have (note that $\phi(t)^{*} \uu=\uu$)
        \begin{equation}\label{eq:Lie-derivative}
          \frac{\partial}{\partial t}\left(\phi(t)^{*} \bt\right) = \phi(t)^{*} \Lie_{\uu} \bt = \Lie_{\uu} \phi(t)^{*} \bt.
        \end{equation}
  \item Given two vector fields $\uu,\vv$ on $M$, we have
        \begin{equation*}
          \Lie_{[\uu,\vv]} \bt = \Lie_{\uu} \Lie_{\vv} \bt - \Lie_{\vv} \Lie_{\uu} \bt .
        \end{equation*}
\end{enumerate}

Consider now a time dependent vector field $\uu(t)$ on $M$. Its flow $\phi(t,s)$ is defined as the solution a time $t$ of the initial value problem
\begin{equation}\label{eq:Cauchy-problem}
  \dot{c}(t) = \uu(t,c(t)), \qquad c(s)=\xx.
\end{equation}
Then, $\phi(t,s)$ is a local diffeomorphism with inverse $\phi(s,t)$ and we have
\begin{equation*}
  \phi(t,s) = \phi(t,\tau) \circ \phi(\tau,s),
\end{equation*}
as soon as the three mappings are defined. The Lie derivative can be extended to time dependent vector fields as follows.
\begin{equation*}
  \Lie_{\uu(t)} \bt := \left.\frac{\partial}{\partial \tau}\right|_{\tau=t} \phi(\tau,t)^{*} \bt .
\end{equation*}

\begin{rem}\label{rem:Lie-path}
  Let $\widetilde{\varphi} = (\varphi(t))$ be a path of diffeomorphism. Its \emph{Eulerian velocity} is defined as the time dependent vector field $\uu(t) := \partial_{t}\varphi \circ \varphi(t)^{-1}$. Given a (time-independent) tensor field $\bt$, we have
  \begin{equation*}
    \frac{\partial}{\partial t} \left(\varphi(t)^{*} \bt \right)=\varphi(t)^{*} (\Lie_{\uu(t)} \bt)
    \qquad \textrm{and} \qquad
    \frac{\partial}{\partial t} \left(\varphi(t)_{*} \bt \right)=-\Lie_{\uu(t)} (\varphi(t)_{*} \bt).
  \end{equation*}
\end{rem}

The preceding results extend to paths of embeddings between two manifolds $\body$ and $M$ and will be summarized by the following lemma.

\begin{lem}\label{lem:fundamental-time-derivative-formula}
  Let $\widetilde{\pp} = (\pp(t))$ be a path of embeddings, $\uu(t) := (\partial_{t}\pp)\circ \pp(t)^{-1}$ be its (right) Eulerian velocity and $\bt(t)$ be a tensor field defined along $\pp(t)$ (\emph{i.e.} on $\Omega_{p(t)} = \pp(t)(\body)$ and possibly time-dependent). Then
  \begin{equation*}
    \partial_{t}(\pp(t)^{*}\bt(t)) = \pp(t)^{*}\left( \partial_{t}\bt + \Lie_{\uu(t)}\bt(t) \right).
  \end{equation*}
\end{lem}

\begin{proof}
  Let $\phi(t,s)$ be the flow of $\uu(t)$. Then we get
  \begin{equation*}
    \pp(s) = \phi(s,t) \circ \pp(t).
  \end{equation*}
  We have thus
  \begin{align*}
    \partial_{t}(\pp(t)^{*}\bt(t)) & = \left.\frac{\partial}{\partial s}\right|_{s=t} (\phi(s,t) \circ \pp(t))^{*} \bt(s)
    \\
                                   & =  \pp(t)^{*} \left.\frac{\partial}{\partial s}\right|_{s=t} \phi(s,t)^{*} \bt(s)
    \\
                                   & =  \pp(t)^{*}\left(\partial_{t}\bt(t) + \Lie_{\uu(t)}\bt(t)\right).
  \end{align*}
\end{proof}

\section{Vector bundles and covariant derivatives}
\label{sec:covariant-derivatives}

The interested reader may consult~\cite{GHL2004,Lan1999,Mic2008,MH1994,SE2020} for complementary points of view on differential geometry. Let $\EE$ be a vector bundle over a manifold $M$. We denote by $\Gamma(\EE)$ the space of smooth sections of $\EE$ and by $\Omega^{k}(M,\EE)$ the space of $k$-forms with values in $\EE$, in other words, the space of sections of the vector bundle $\Lambda^{k}T^{\star}M \otimes \EE$ (in particular, $\Omega^{0}(M,\EE) = \Gamma(\EE)$).

\begin{defn}[Covariant derivative]
  A \emph{covariant derivative} on vector bundle $\EE$ over $M$ is a linear operator
  \begin{equation*}
    \nabla : \Gamma(\EE) \to \Omega^{1}(M,\EE), \qquad \bs \mapsto \nabla \bs ,
  \end{equation*}
  which satisfies the \emph{Leibniz identity}.
  \begin{equation*}
    \nabla (f\bs) = df \otimes \bs + f \, \nabla \bs,
  \end{equation*}
  for any function $f \in \Cinf(M)$ and any section $\bs \in \Gamma(\EE)$.
\end{defn}

\begin{rem}
  The set of all covariant derivatives defined on a given vector bundle $\EE$ has an affine structure. Indeed, given two covariant derivatives $\nabla^{1}$ and $\nabla^{2}$, the difference $\nabla^{2}-\nabla^{1}$ is a section of the vector bundle
  \begin{equation*}
    (T^{\star}M \otimes \EE^{\star}) \otimes \EE.
  \end{equation*}
  Hence, this set is an affine space with associated vector space $\Gamma((T^{\star}M \otimes \EE^{\star}) \otimes \EE)$.
\end{rem}

\begin{rem}\label{rem:canonical-covariant-derivative}
  If a vector bundle $\EE$ with base manifold $M$ and fiber type $E$ (a vector space) is trivializable, meaning that there exists a vector bundle isomorphism
  \begin{equation*}
    \Psi : \EE \to M \times E, \qquad \vv_{x} \mapsto (x,\vv),
  \end{equation*}
  then, each section $\bs$ of $\EE$ corresponds bijectively to a \emph{vector valued function} $S$ defined by
  \begin{equation*}
    S: M \to E, \qquad x \mapsto p_{2} \circ \Psi(\bs(x)),
  \end{equation*}
  where $p_{2}: M \times E$ is the projection onto the second factor. Therefore, there is a canonical covariant derivative associated with this trivialization which is given by
  \begin{equation*}
    (\nabla_{X} \bs)(x) := \Psi^{-1}(x, d_{x}S.X).
  \end{equation*}
\end{rem}

\begin{defn}[Curvature]\label{eq:defcourbure}
  Given a \emph{covariant derivative} $\nabla$ on a vector bundle $\EE$, its \emph{curvature} is the mapping
  \begin{equation*}
    R : \Gamma(\EE) \to \Omega^{2}(M,\EE)
  \end{equation*}
  defined by
  \begin{equation*}
    R(X,Y)\bs := \nabla_{X}\nabla_{Y}\bs - \nabla_{Y}\nabla_{X}\bs - \nabla_{[X,Y]}\bs.
  \end{equation*}
\end{defn}

When $\EE = TM$ is the tangent bundle of a manifold $M$, we define the \emph{torsion} of this covariant derivative, by the formula
\begin{equation*}
  T(X,Y) := \nabla_{X} Y - \nabla_{Y} X - [X, Y], \qquad  X, Y \in \Vect(M),
\end{equation*}
which is a mixed tensor field of type $(1,2)$. The curvature tensor of $\nabla$ is a mixed tensor field of type $(1,3)$, which writes
\begin{equation*}
  R(X,Y)Z = \nabla_{X}\nabla_{Y}Z - \nabla_{Y}\nabla_{X}Z - \nabla_{[X,Y]}Z, \qquad  X, Y, Z \in \Vect(M).
\end{equation*}

\begin{defn}
  A covariant derivative on the tangent bundle $TM$ of a manifold $M$  is \emph{symmetric} if its torsion is zero, that is if
  \begin{equation*}
    \nabla_{\vv} \ww - \nabla_{\ww} \vv = [\vv, \ww], \qquad  \forall \vv, \ww \in \Vect(M),
  \end{equation*}
  where $[\vv, \ww]:=\Lie_{\vv} \ww$ is the Lie bracket of the vector fields $\vv$ and $\ww$.
\end{defn}

\begin{rem}
  One can show the existence on any differential manifold $M$ of a covariant derivative. However, there are an infinite number of such derivatives and none of them play a particular role. On the other hand, if a manifold $M$ has a Riemannian metric $g$, then there is a unique symmetric covariant derivative $\nabla$ such that $\nabla g = 0$ (see for example~\cite[Theorem~2.51]{GHL2004}), this is the \emph{Riemannian covariant derivative}.
\end{rem}

Any covariant derivative on $TM$ induces by the Leibniz rule a covariant derivative on all tensor bundles of $M$. The link between the Lie derivative and a \emph{symmetric} covariant derivative is then recalled in the following theorem.

\begin{prop}\label{prop:Lie-derivative-versus-covariant-derivative}
  Let $M$ be a differential manifold with a \emph{symmetric} covariant derivative $\nabla$. Then we have the following relations:
  \begin{enumerate}
    \item Lie derivative of a function $f$:
          \begin{equation*}
            \Lie_{\uu} f = \nabla_{\uu} f = df.\uu;
          \end{equation*}
    \item Lie derivative of a vector field $\ww = (w^{i})$:
          \begin{equation*}
            \Lie_{\uu} \ww = [\uu, \ww] = \nabla_{\uu} \ww - \nabla_{\ww} \uu;
          \end{equation*}
    \item Lie derivative of a covector field ($1$-form) $\alpha = (\alpha_{i})$:
          \begin{equation*}
            \Lie_{\uu} \alpha = \nabla_{\uu} \alpha + (\nabla \uu)^{\star} \alpha;
          \end{equation*}
    \item Lie derivative of a second-order covariant tensor field, $\bk = (k_{ij})$:
          \begin{equation*}
            \Lie_{\uu} \bk = \nabla_{\uu} \bk + (\nabla \uu)^{\star} \bk + \bk  (\nabla \uu);
          \end{equation*}
    \item Lie derivative of a second-order contravariant tensor field, $\btau = (\tau^{ij})$:
          \begin{equation*}
            \Lie_{\uu} \btau = \nabla_{\uu} \btau - (\nabla \uu) \btau - \btau (\nabla \uu)^{\star};
          \end{equation*}
    \item Lie derivative of a second-order mixed tensor field, $\hat{\bt} = (\tensor{\hat{t}}{^{i}}{_{j}})$:
          \begin{equation*}
            \Lie_{\uu} \hat{\bt} = \nabla_{\uu} \hat{\bt} - (\nabla \uu) \hat{\bt} + \hat{\bt} (\nabla \uu);
          \end{equation*}
    \item Lie derivative of a second-order mixed tensor field, $\check{\bt} = (\tensor{\check{t}}{_{i}}{^{j}})$:
          \begin{equation*}
            \Lie_{\uu} \check{\bt} = \nabla_{\uu} \check{\bt} + (\nabla \uu)^{\star} \check{\bt} - \check{\bt} (\nabla \uu)^{\star}.
          \end{equation*}
  \end{enumerate}
\end{prop}

In order to intrinsically define the geodesic equation of a Riemannian manifold, but also the covariant derivative of the Lagrangian velocity, it is necessary to extend the notion of covariant derivative to vector fields (and more generally tensor fields) which are only defined along a mapping $f: K \to M$, between two manifolds $K$ and $M$. The rigorous formulation of such a definition requires first to introduce the notion of \emph{pull-back of a vector bundle}~\cite{Hus1994}.

\begin{defn}\label{def:pullback-bundle}
  Let $\pi : \EE \to M$ be a vector bundle and $f: K \to M$ be a smooth mapping. Then the set
  \begin{equation*}
    f^{*}\EE := \bigsqcup_{k \in K} E_{f(k)} \subset \EE
  \end{equation*}
  is a vector bundle above $K$, referred to as the \emph{pull-back} by $f$ of the vector bundle $\EE$. A section of this bundle is therefore a mapping $s : K \to \EE$, such as $\pi (s(k)) = f(k)$.
\end{defn}

\begin{exam}
  A \emph{vector field defined along a curve} $c : I \to M$ is a curve $X : I \to TM$, such that $X(t) \in T_{c(t)}M$, for any $t \in I$.
\end{exam}

\begin{exam}
  The Lagrangian velocity $\VV(t) : \body \to T\espace$, at time $t$, is a section of the pullback bundle $\pp(t)^{*}T\espace$.
\end{exam}

\begin{prop}\label{prop:pullback-connection}
  Let $\pi : \EE \to M$ be a vector bundle, equipped with a covariant derivative $\nabla$ and $f: K \to M$,
  a smooth mapping. Then, there exists a unique covariant derivative, denoted $f^{*}\nabla$, on the vector bundle $f^{*}\EE$, called the \emph{pull-back of $\nabla$}, and such that
  \begin{equation*}
    (f^{*}\nabla)_{X} (f \circ s) = f^{*}\left( \nabla_{Tf.X} s\right),
  \end{equation*}
  for any section $s$ of $\EE$ and any vector field $X$ on $K$.
\end{prop}

\begin{exam}
  Consider the case where $K = \body$, $M=\espace$, $\EE = T\espace$ is the tangent vector bundle to $\espace$ and
  \begin{equation*}
    f = p: \body \to \espace,
  \end{equation*}
  is an embedding of $\body$ in $\espace$. Consider a local coordinate system $(X^{I})$ on $\body$ and a local coordinate system $(x^{i})$ on $\espace$, where Christoffel's symbols are written $\Gamma_{ij}^{k}$. Let $\VV$ be a section of $p^{*}T\espace$, \emph{i.e.} a mapping
  \begin{equation*}
    \VV : \body \to T\espace, \quad \text{such that} \quad \pi (X(X)) = p(X),
  \end{equation*}
  then, we have
  \begin{equation*}
    [(p^{*}\nabla)_{\partial_{X^{I}}} V]^{k} = \partial_{X^{I}} V^{k} + \Gamma_{ij}^{k} \frac{\partial p^{i}}{\partial_{X^{I}}} V^{j}.
  \end{equation*}
\end{exam}

\begin{exam}
  Let $(M,\bgamma)$ be a Riemannian manifold and $\nabla$ the associated covariant derivative. Let $c : I \to M$ be curve. Then, the pull-back of $\nabla$ by $c$, usually noted $D_{t}$ is defined on the vector space of vector fields defined along $c$. It is characterized by the following properties:
  \begin{enumerate}
    \item for any function $f: I \to \RR$,
          \begin{equation*}
            D_{t}(fX)(t) = f^{\prime}(t)\, X(t) + f(t)\, D_{t}(X)(t);
          \end{equation*}
    \item If $X(t) = \tilde{X}(c(t))$ where $\tilde{X}$ is a vector field on $M$, then
          \begin{equation*}
            (D_{t}X)(t) = (\nabla_{c^{\prime}(t)}\tilde{X})((t)).
          \end{equation*}
  \end{enumerate}
  In a local coordinate system $(x^{i})$ of $M$, we have:
  \begin{equation*}
    (D_{t}X)^{k} = \partial_{t}X^{k} + \Gamma_{ij}^{k} (\partial_{t}x^{i})X^{j}.
  \end{equation*}
\end{exam}

\begin{rem}
  On an infinite dimensional manifold, equipped with a covariant derivative, the curvature is defined along a parameterized surface $c(s,t)$, and writes
  \begin{equation*}
    R(\partial_s, \partial_{t})X = D_{s}D_{t}X - D_{t}D_{s}X .
  \end{equation*}
\end{rem}


\noindent \textbf{Acknowledgements.} It is a pleasure to thank Emmanuelle Rouhaud for stimulating discussions concerning the concept of objective derivative and related to her own work~\cite{RPK2013,PRA2015}. We would also like to thank Boris Desmorat for the day-long scientific discussions that helped initiate the present work.


\end{document}